\documentclass[a4paper,10pt]{article}
\usepackage[latin1]{inputenc}
\usepackage[T1]{fontenc}
\usepackage{lmodern}
\usepackage{a4wide}
\usepackage{empheq}
\usepackage{amssymb,amsthm}
\usepackage{enumerate,bbm}
\usepackage[a4paper]{hyperref}

\newcommand{\m}{\mbox}
\newcommand{\eps}{\varepsilon}
\renewcommand{\leq}{\leqslant}
\renewcommand{\geq}{\geqslant}
\DeclareMathOperator{\re}{Re}
\DeclareMathOperator{\im}{Im}
\newcommand{\px}{\partial_x}
\newcommand{\phib}{\varphi}

\newcommand{\vers}{\longrightarrow}

\newcommand{\cvf}{\rightharpoonup}

\newcommand{\lent}{[\kern-0.15em[}
\newcommand{\rent}{]\kern-0.15em]}
\newcommand{\ient}[2]{\lent #1,#2\rent}
\newcommand{\unn}{\ient{1}{N}}

\newcommand{\nh}[1]{{\|#1\|}_{H^1}}

\newcommand{\nld}[1]{{\|#1\|}_{L^2}}

\newcommand{\nli}[1]{{\|#1\|}_{L^{\infty}}}

\newcommand{\carre}[1]{{#1}^2}
\newcommand{\Carre}[1]{{\left( #1 \right)}^2}

\newcommand{\wt}[1]{\widetilde{#1}}

\renewcommand{\a}{\mathfrak{a}^-}
\renewcommand{\b}{\mathfrak{b}}
\newcommand{\g}{\gamma}
\newcommand{\alpham}{\boldsymbol{\alpha}^-}

\newcommand{\ta}{T(\a)}
\newcommand{\an}{A_1,\ldots,A_N}

\newcommand{\dt}{\frac{d}{dt}}
\newcommand{\zt}{\wt z}

\newcommand{\C}{\mathbb{C}}
\renewcommand{\H}{\mathcal{H}}
\renewcommand{\L}{\mathcal{L}}
\newcommand{\M}{\mathcal{M}}
\newcommand{\N}{\mathcal{N}}
\newcommand{\R}{\mathbb{R}}
\renewcommand{\S}{\mathbb{S}}

\renewcommand{\Bar}[1]{\overline{#1}}
\newcommand{\cmin}{c_{\mathrm{min}}}
\newcommand{\emin}{e_{\mathrm{min}}}

\theoremstyle{plain}
\newtheorem{theo}{Theorem}[section]
\newtheorem{prop}[theo]{Proposition}
\newtheorem{lem}[theo]{Lemma}
\newtheorem{cor}[theo]{Corollary}
\newtheorem{claim}[theo]{Claim}
\theoremstyle{definition}
\newtheorem{defi}[theo]{Definition}
\newtheorem{rem}[theo]{Remark}
\newtheorem{notation}[theo]{Notation}

\allowdisplaybreaks[1]
\numberwithin{equation}{section}

\title{Multi-existence of multi-solitons for the supercritical nonlinear Schrödinger equation in one dimension}
\author{Vianney Combet}
\date{Universit\'e de Versailles Saint-Quentin-en-Yvelines, \\
Laboratoire de Math\'ematiques de Versailles, UMR CNRS 8100, \\
 45, av. des \'Etats-Unis,
 78035 Versailles Cedex, France\\
 vianney.combet@math.uvsq.fr}

\begin{document}

\maketitle

\begin{abstract}
For the $L^2$ supercritical generalized Korteweg-de Vries equation, we proved in \cite{combet:multisoliton} the existence and uniqueness of an $N$-parameter family of $N$-solitons. Recall that, for any $N$ given solitons, we call $N$-soliton a solution of the equation which behaves as the sum of these $N$ solitons asymptotically as $t\to+\infty$. In the present paper, we also construct an $N$-parameter family of $N$-solitons for the supercritical nonlinear Schrödinger equation, in dimension $1$ for the sake of simplicity. Nevertheless, we do not obtain any classification result; but recall that, even in subcritical and critical cases, no general uniqueness result has been proved yet.
\end{abstract}

\section{Introduction}

\subsection{The nonlinear Schrödinger equation} \label{subsec:NLS}

We consider the $L^2$ supercritical focusing nonlinear Schrödinger equation in one dimension: \begin{equation} \label{eq:NLS} \tag{NLS} \begin{cases} i\partial_t u+\px^2 u +|u|^{p-1}u=0,\\ u(0)=u_0\in H^1(\R), \end{cases} \end{equation} where $(t,x)\in\R^2$, $p>5$ is real, and $u$ is a complex-valued function. Recall first that Ginibre and Velo \cite{ginibrevelo} proved that \eqref{eq:NLS} is locally well-posed in $H^1(\R)$ for $p>1$: for any $u_0\in H^1(\R)$, there exist $T>0$ and a unique maximal solution $u\in C([0,T),H^1(\R))$ of \eqref{eq:NLS}. Moreover, either $T=+\infty$ or $T<+\infty$ and then $\lim_{t\to T} \nld{\px u(t)}=+\infty$. It is also well-known that $H^1$ solutions of \eqref{eq:NLS} satisfy the following three conservation laws: for all $t\in [0,T)$, \begin{gather*} M(u(t)) = \int |u(t)|^2 = M(u_0)\quad \m{(mass)},\\ E(u(t)) = \frac{1}{2}\int |\px u(t)|^2 -\frac{1}{p+1}\int |u(t)|^{p+1} = E(u_0)\quad \m{(energy)},\\ P(u(t)) = \im\int \px u(t)\bar u(t) = P(u_0)\quad \m{(momentum)}. \end{gather*}

Recall also that \eqref{eq:NLS} admits the following symmetries.
\begin{itemize}
\item Space-time translation invariance: if $u(t,x)$ satisfies \eqref{eq:NLS}, then for any $t_0,x_0\in\R$, $w(t,x)=u(t-t_0,x-x_0)$ also satisfies \eqref{eq:NLS}.
\item Scaling invariance: if $u(t,x)$ satisfies \eqref{eq:NLS}, then for any $\lambda>0$, $w(t,x)=\lambda^{\frac{2}{p-1}} u(\lambda^2 t,\lambda x)$ also satisfies \eqref{eq:NLS}.
\item Phase invariance: if $u(t,x)$ satisfies \eqref{eq:NLS}, then for any $\g_0\in\R$, $w(t,x)=u(t,x)e^{i\g_0}$ also satisfies \eqref{eq:NLS}.
\item Galilean invariance: if $u(t,x)$ satisfies \eqref{eq:NLS}, then for any $v_0\in\R$, $w(t,x)=u(t,x-v_0t)e^{i(\frac{v_0}{2}x-\frac{v_0^2}{4}t)}$ also satisfies \eqref{eq:NLS}.
\end{itemize}

We now consider solitary waves of \eqref{eq:NLS}, in other words solutions of the form $u(t,x)=e^{ic_0t}Q_{c_0}(x)$, where $c_0>0$ and $Q_{c_0}$ is solution of \begin{equation} \label{eq:Qc} Q_{c_0}>0,\quad Q_{c_0}\in H^1(\R),\quad Q''_{c_0}+Q_{c_0}^p=c_0Q_{c_0}. \end{equation} Recall that such positive solution of \eqref{eq:Qc} exists and is unique up to translations, and is moreover the solution of a variational problem: we call $Q_{c_0}$ the solution of \eqref{eq:Qc} which is even, and we denote $Q:=Q_1$. By the symmetries of \eqref{eq:NLS}, for any $\g_0,v_0,x_0\in\R$, \[ R_{c_0,\g_0,v_0,x_0}(t,x) = Q_{c_0}(x-v_0t-x_0)e^{i(\frac{v_0}{2}x-\frac{v_0^2}{4}t+c_0t+\g_0)}\] is also a solitary wave of \eqref{eq:NLS}, moving on the line $x=v_0t+x_0$, that we also call \emph{soliton}.

Finally recall that, in the supercritical case $p>5$, solitons are \emph{unstable} (see \cite{gss}). A striking illustration of this fact is the following result of Duyckaerts and Roudenko \cite{duy} (adapted from a previous work of Duyckaerts and Merle \cite{duymerle}), obtained for the 3d focusing cubic nonlinear Schrödinger equation (NLS-3d), which is also $L^2$ supercritical and $H^1$ subcritical as in our case.
\begin{prop}[\cite{duy}] \label{th:duyrou}
Let $A\in\R$. If $t_0=t_0(A)>0$ is large enough, then there exists a radial solution $U^A\in C^{\infty}([t_0,+\infty),H^{\infty})$ of (NLS-3d) such that \[ \forall b\in\R, \exists C>0, \forall t\geq t_0,\quad {\| U^A(t)-e^{it}Q-Ae^{(i-e_0)t}Y^+\|}_{H^b}\leq Ce^{-2e_0t},\] where $e_0>0$ and $Y^+\neq 0$ is in the Schwartz space $\mathcal{S}$.
\end{prop}

In particular, $U^A(t)\neq e^{it}Q$ if $A\neq 0$, whereas $\lim_{t\to+\infty} \nh{U^A(t)-e^{it}Q}=0$. Note that, in the subcritical and critical cases $p\leq 5$, no such special solutions $U^A(t)$ can exist, due to a variational characterization of $Q$. Indeed, if $\lim_{t\to+\infty} \nh{u(t)-e^{it}Q}=0$, then $u(t)=e^{it}Q$ in this case. The purpose of this paper is to extend Proposition \ref{th:duyrou} to multi-solitons.

\subsection{Multi-solitons}

Now, we focus on  multi-soliton solutions. Given $4N$ parameters defining $N\geq 2$ solitons with different speeds, \begin{equation} \label{parametersNLS} v_1<\cdots<v_N, \quad c_1,\ldots,c_N\in\R_+^*, \quad \g_1,\ldots,\g_N\in\R, \quad x_1,\ldots,x_N\in\R, \end{equation} we set \[ R_j(t) = R_{c_j,\g_j,v_j,x_j}(t)\quad \m{and}\quad R(t)=\sum_{j=1}^N R_j(t), \] and we call $N$-soliton a solution $u(t)$ of \eqref{eq:NLS} such that \[ \nh{u(t)-R(t)} \vers 0 \quad \text{as} \quad t\to+\infty. \]

Let us recall known results on multi-solitons.
\begin{itemize}
\item In the $L^2$ subcritical and critical cases, \emph{i.e.} for \eqref{eq:NLS} with $p\leq 5$, there exists a large literature on the problem of existence of multi-solitons and on their properties. Merle \cite{merle:kblowup} first established an existence result in the critical case, as a consequence of a blow up result and the conformal invariance. This result was extended by Martel and Merle \cite{martel:NLS} to the subcritical case, using arguments developed by Martel, Merle and Tsai \cite{martel:tsaiNLS} for the stability in $H^1$ of solitons. Nevertheless, we recall that no general uniqueness result has been proved, contrarily to the generalized Korteweg-de Vries (gKdV) equation (see \cite{martel:Nsoliton}).

    For other stability and asymptotic stability results on multi-solitons of some nonlinear Schrödinger equations, see \cite{perelman,perelmanbis,rss}.
\item In the $L^2$ supercritical case, \emph{i.e.} in a situation where solitons are known to be unstable, Côte, Martel and Merle \cite{martel:Nsolitons} have recently proved the existence of at least one multi-soliton solution for \eqref{eq:NLS}:
\end{itemize}

\begin{theo}[\cite{martel:Nsolitons}] \label{th:cmmNLS}
Let $p>5$ and $N\geq 2$. Let $v_1<\cdots<v_N$, $(c_1,\ldots,c_N)\in(\R_+^*)^N$, $(\g_1,\ldots,\g_N)\in\R^N$ and $(x_1,\ldots,x_N)\in\R^N$. There exist $T_0 \in\R$, $C,\sigma_0>0$, and a solution $\phib \in C([T_0,+\infty), H^1)$ of \eqref{eq:NLS} such that \[ \forall t \in [T_0,+\infty), \quad \nh{\phib(t)-R(t)} \leq C e^{-\sigma_0^{3/2} t}. \]
\end{theo}

Recall that, with respect to \cite{martel:NLS,martel:tsaiNLS}, the proof of Theorem \ref{th:cmmNLS} relies on an additional topological argument to control the unstable nature of the solitons. Finally, recall that Theorem \ref{th:cmmNLS} was also obtained for the $L^2$ supercritical gKdV equation, and has been a crucial starting point in \cite{combet:multisoliton} to obtain the multi-existence and the classification of multi-solitons. It is a similar multi-existence result that we propose to prove in this paper.

\subsection{Main result and outline of the paper}

The whole paper is devoted to prove the following theorem of existence of a family of multi-solitons for the supercritical \eqref{eq:NLS} equation.

\begin{theo} \label{th:mainNLS}
Let $p>5$, $N\geq 2$, $v_1<\cdots<v_N$, $(c_1,\ldots,c_N)\in(\R_+^*)^N$, $(\g_1,\ldots,\g_N)\in\R^N$ and $(x_1,\ldots,x_N)\in\R^N$. Denote $R=\sum_{j=1}^N R_{c_j,\g_j,v_j,x_j}$.

Then there exist $\g>0$ and an $N$-parameter family ${(\phib_{\an})}_{(\an)\in\R^N}$ of solutions of~\eqref{eq:NLS} such that, for all $(\an)\in\R^N$, there exist $C>0$ and $t_0>0$ such that \[\forall t\geq t_0,\ \nh{\phib_{\an}(t)-R(t)}\leq Ce^{-\g t}, \] and if $(A'_1,\ldots,A'_N)\neq (\an)$, then $\phib_{A'_1,\ldots,A'_N} \neq \phib_{\an}$.
\end{theo}

\begin{rem}
As underlined above, the question of the classification of multi-solitons is open for the \eqref{eq:NLS} equation, even in the subcritical case, while it was obtained in \cite{combet:multisoliton} for the supercritical gKdV equation, and in \cite{martel:Nsoliton} for the subcritical and critical cases. Although we expect that the family constructed in Theorem \ref{th:mainNLS} characterizes all multi-solitons, the lack of monotonicity properties such as for the gKdV equation does not allow to prove it for now.
\end{rem}

The paper is organized as follows. In the next section, we briefly recall some well-known results on multi-solitons and on the linearized equation. One of the most important facts about the linearized equation, also strongly used in \cite{duy,martel:Nsolitons}, is the determination of the spectrum of the linearized operator $\L$ around the soliton $e^{it}Q$ (proved in \cite{weinstein:modulational} and \cite{grillakis}): $\sigma(\L)\cap\R = \{-e_0,0,+e_0\}$ with $e_0>0$, and moreover $e_0$ and $-e_0$ are simple eigenvalues of $\L$ with eigenfunctions $Y^+$ and $Y^-$. Indeed, $Y^{\pm}$ allow to control the negative directions of the linearized energy around a soliton (see Proposition \ref{th:coer}). Moreover, by a simple scaling argument, we determine the eigenvalues of the linearized operator around $e^{ic_jt}Q_{c_j}$, and in particular $\pm e_j=\pm c_j^{3/2}e_0$ are simple eigenvalues with eigenfunctions $Y_j^{\pm}$ (see Notation \ref{RjNLS} for precise definitions).

In Section \ref{sec:constructionNLS}, we construct the family $(\phib_{\an})$ described in Theorem \ref{th:mainNLS}. To do this, we first claim Proposition \ref{th:princNLS}, which is the key point of the proof of the multi-existence result as in~\cite{combet:multisoliton}, and can be summarized as follows. \textit{Let $\phib$ be a multi-soliton given by Theorem \ref{th:cmmNLS}, $j\in\unn$ and $A_j\in\R$. Then there exists a solution $u(t)$ of \eqref{eq:NLS} such that \[ \nh{u(t)-\phib(t)-A_je^{-e_jt}Y_j^+(t)} \leq e^{-(e_j+\g)t}, \] for $t$ large and for some small $\g>0$.} This means that, similarly as in \cite{duy} for one soliton, we can perturb the multi-soliton $\phib$ locally around \emph{one} given soliton at the order $e^{-e_jt}$. Since it is not significant to perturb $\phib$ at order $e_j$ before order $e_k$ if $e_j>e_k$, the construction of $\phib_{\an}$ has to be done following values (possibly equal) of $e_j$.

Finally, to prove Proposition \ref{th:princNLS}, we follow the strategy of the proof of the similar proposition in \cite{combet:multisoliton}, except for the monotonicity property of the energy which does not hold for the \eqref{eq:NLS} equation. If this property of monotonicity was necessary to obtain the classification, we prove that a slightly different functional estimated regardless its sign is sufficient to reach our purpose. We also rely on refinements of arguments developed in \cite{martel:Nsolitons}, in particular the topological argument to control the unstable directions.

\section{Preliminary results}

\begin{notation}
They are available in the whole paper.
\begin{enumerate}[(a)]
\item We denote $\px v = v_x$ the partial derivative of $v$ with respect to $x$.
\item For $h\in\C$, we denote $h_1=\re h$ and $h_2=\im h$.
\item For $f,g\in L^2$, $(f,g)=\re\int f\bar g$ denotes the real scalar product.
\item The Sobolev space $H^s$ is defined by $H^s(\R) = \{ u\in \mathcal{D}'(\R)\ |\ {(1+\xi^2)}^{s/2}\hat{u}(\xi) \in L^2(\R) \}$, and in particular $H^1(\R) = \{ u\in L^2(\R)\ |\ \nh{u}^2 = \nld{u}^2 +\nld{\px u}^2 <+\infty \} \hookrightarrow L^{\infty}(\R)$.
\item If $a$ and $b$ are two functions of $t$ and if $b$ is positive, we write $a=O(b)$ when there exists a constant $C>0$ independent of $t$ such that $|a(t)|\leq Cb(t)$ for all $t$.
\end{enumerate}
\end{notation}

\subsection{Linearized operator around a stationary soliton}

The linearized equation appears if one considers a solution of \eqref{eq:NLS} close to the soliton $e^{it}Q$. More precisely, if $u(t,x)=e^{it}(Q(x)+h(t,x))$ satisfies \eqref{eq:NLS}, then $h$ satisfies $\partial_t h+\L h = O(h^2)$, where the operator $\L$ is defined for $v=v_1+iv_2$ by \[\L v = -L_-v_2+iL_+v_1,\] and the self-adjoint operators $L_+$ and $L_-$ are defined by \[ L_+v_1 = -\px^2v_1+v_1-pQ^{p-1}v_1,\quad L_-v_2 = -\px^2v_2+v_2-Q^{p-1}v_2. \] The spectral properties of $\L$ are well-known (see \cite{grillakis,weinstein:modulational} for instance), and summed up in the following proposition.

\begin{prop}[\cite{grillakis,weinstein:modulational}] \label{th:spectrumNLS}
Let $\sigma(\L)$ be the spectrum of the operator $\L$ defined on $L^2(\R)\times L^2(\R)$ and let $\sigma_{\mathrm{ess}}(\L)$ be its essential spectrum. Then \[ \sigma_{\mathrm{ess}}(\L) = \{i\xi\ ;\ \xi\in\R, |\xi|\geq 1\},\quad \sigma(\L)\cap\R = \{-e_0,0,+e_0\}\quad \m{with}\ e_0>0. \] Furthermore, $e_0$ and $-e_0$ are simple eigenvalues of $\L$ with eigenfunctions $Y^+$ and $Y^-=\Bar{Y^+}$ which have an exponential decay at infinity. Finally, the null space of $\L$ is spanned by $\px Q$ and $iQ$, and as a consequence, the null space of $L_+$ is spanned by $\px Q$ and the null space of $L_-$ is spanned by~$Q$.
\end{prop}

\begin{rem} \label{th:Ydecroit}
By standard ODE techniques, we can quantify the exponential decay of $Y^{\pm}$ and $\px Y^{\pm}$ at infinity. In fact, there exist $\eta_0>0$ and $C>0$ such that, for all $x\in\R$, \[ |Y^{\pm}(x)| +|\px Y^{\pm}(x)|\leq Ce^{-\eta_0 |x|}. \]
\end{rem}

Moreover, $\L$, $L_+$ and $L_-$ satisfy some properties of positivity or coercivity. The following proposition sums up the two properties useful for our purpose. Note that the first one is proved in \cite{weinstein:modulational}, while the second one is proved in \cite{duymerle,duy}.

\begin{prop}[\cite{weinstein:modulational,duy}] \label{th:coer}
\begin{enumerate}[(i)]
\item For all $f\in H^1\setminus \{\lambda Q\ ;\ \lambda\in\R\}$ real-valued, one has $\int (L_-f)f>0$.
\item There exists $\kappa_0>0$ such that, for all $v=v_1+iv_2\in H^1$, \begin{multline} \label{eq:coer} (L_+v_1,v_1)+(L_-v_2,v_2)\geq \frac{1}{\kappa_0}\nh{v}^2 -\kappa_0\left[ \Carre{\int \px Qv_1}+ \Carre{\int Qv_2}\right. \\ \left. +\Carre{\im\int Y^+\bar v}+ \Carre{\im\int Y^-\bar v}\right].  \end{multline}
\end{enumerate}
\end{prop}

Finally, we extend Proposition \ref{th:spectrumNLS} to the operator $\L_c$ linearized around a soliton $e^{ict}Q_c(x)$, by a simple scaling argument. In fact, we recall that if $u$ is a solution of \eqref{eq:NLS}, then $w(t,x)=\lambda^{\frac{2}{p-1}} u(\lambda^2 t,\lambda x)$ is also a solution, and moreover, we have $Q_c(x)=c^{\frac{1}{p-1}}Q(\sqrt cx)$ for all $c>0$.

\begin{cor} \label{th:spectrumcNLS}
Let $c>0$. For $v=v_1+iv_2$, $\L_c$ is defined by $\L_c v=-L_{c-}v_2+iL_{c+}v_1$, where \[ L_{c+}v_1 = -\px^2 v_1+cv_1-pQ_c^{p-1}v_1\quad \m{and}\quad L_{c-}v_2 = -\px^2v_2+cv_2-Q_c^{p-1}v_2. \] Moreover, the spectrum $\sigma(\L_c)$ of $\L_c$ satisfies \[ \sigma(\L_c)\cap\R = \{-e_c,0,+e_c\},\ \m{where}\ e_c=c^{3/2}e_0>0.\] Finally, $e_c$ and $-e_c$ are simple eigenvalues of $\L_c$ with eigenfunctions $Y_c^+$ and $Y_c^-$, where \[ Y_c^+(x) = c^{1/4}Y^+(\sqrt cx)\quad \m{and}\quad Y_c^- = \Bar{Y_c^+},\] and the null space of $\L_c$ is spanned by $\px Q_c$ and $iQ_c$.
\end{cor}

\begin{claim} \label{th:normalization}
One can normalize $Y^{\pm}$ so that \begin{equation} \label{eq:normalization} -\im\int \carre{(Y^+)} = 1 \quad \m{and still}\quad Y^- = \Bar{Y^+}. \end{equation}
\end{claim}

\begin{proof}
Denote $Y_1=\re Y^+$ and $Y_2=\im Y^+$. Thus, we have $Y^+ = Y_1+iY_2$, $Y^-=Y_1-iY_2$, and \[ L_+Y_1 = e_0Y_2,\ L_-Y_2 = -e_0Y_1. \] Now, suppose that there exists $\lambda\in\R$ such that $Y_2=\lambda Q$. Then, we would have $L_-Y_2=-e_0Y_1=\lambda L_-Q = 0$, and so $Y_1=0$. But it would imply $L_+Y_1=0=e_0Y_2$, and so $Y_2=0$, which would be a contradiction. Therefore, by (i) of Proposition \ref{th:coer}, we have $\int (L_-Y_2)Y_2 = -e_0\int Y_1Y_2>0$. Hence, since $\im \int \carre{(Y^+)} = 2\int Y_1Y_2$, we normalize $Y^{\pm}$ by taking \[ \wt{Y^+} = \frac{Y^+}{\sqrt{-2\int Y_1Y_2}},\ \wt{Y^-} = \Bar{\wt{Y^+}}.\qedhere \]
\end{proof}

\subsection{Multi-solitons results}

A set of parameters \eqref{parametersNLS} being given, we adopt the following notation.

\begin{notation} \label{RjNLS}
For all $j\in \unn$, define:
\begin{enumerate}[(i)]
\item $\lambda_j(t,x)=x-v_jt-x_j$ and $\theta_j(t,x)=\frac{1}{2}v_jx-\frac{1}{4}v_j^2t+c_jt+\g_j$.
\item $R_j(t,x)=Q_{c_j}(\lambda_j(t,x))e^{i\theta_j(t,x)}$, where $Q_c(x)=c^{\frac{1}{p-1}}Q(\sqrt cx)$.
\item $Y_j^{\pm}(t,x)=Y_{c_j}^{\pm}(\lambda_j(t,x))e^{i\theta_j(t,x)}$, where $Y_c^{\pm}(x)=c^{1/4}Y^{\pm}(\sqrt cx)$.
\item $e_j=e_{c_j}$, where $e_c=c^{3/2}e_0$.
\end{enumerate}
\end{notation}

Now, to estimate interactions between solitons, we denote $\cmin = \min\{ c_k\ ;\ k\in\unn\}$, and the small parameters \begin{equation} \label{gammaNLS} \sigma_0 = \min \{ \eta_0\sqrt{\cmin},e_0^{2/3}\cmin,\cmin,v_2-v_1,\ldots,v_N-v_{N-1} \} \quad \m{and} \quad  \g = \frac{\sigma_0^{3/2}}{10^6}. \end{equation}

From \cite{martel:NLS}, it appears that $\g$ is a suitable parameter to quantify interactions between solitons in large time. For instance, we have, for $j\neq k$ and all $t\geq 0$, \begin{equation} \label{eq:interactNLS} \int |R_j(t)||R_k(t)| + |(R_j)_x(t)||(R_k)_x(t)| \leq Ce^{-10\g t}. \end{equation} From the definition of $\sigma_0$ and Remark \ref{th:Ydecroit}, such an inequality is also true for $Y_j^{\pm}$.

Moreover, since $\sigma_0$ has the same definition as in \cite{martel:Nsolitons}, Theorem \ref{th:cmmNLS} can be rewritten as follows. \emph{There exist $T_0\in\R$, $C>0$ and $\phib \in C([T_0,+\infty),H^1)$ such that, for all $t\geq T_0$, \begin{equation} \label{eq:cmmNLS} \nh{\phib(t)-R(t)} \leq Ce^{-4\g t}. \end{equation} }

\section{Construction of a family of multi-solitons} \label{sec:constructionNLS}

In this section, we prove Theorem \ref{th:mainNLS} as a consequence of the following crucial Proposition \ref{th:princNLS}. Let $p>5$, $N\geq 2$, and a set of parameters \eqref{parametersNLS}. Denote $R=\sum_{k=1}^N R_k$ and $\phib$ a multi-soliton solution satisfying \eqref{eq:cmmNLS}, as defined in Theorem \ref{th:cmmNLS} for example.

\begin{prop} \label{th:princNLS}
Let $j\in \unn$ and $A_j\in\R$. Then there exist $t_0>0$ and $u \in C([t_0,+\infty),H^1)$ a solution of \eqref{eq:NLS} such that \begin{equation} \label{eq:perturbNLS} \forall t\geq t_0,\quad \nh{u(t)-\phib(t)-A_je^{-e_jt}Y_j^+(t)} \leq e^{-(e_j+\g)t}. \end{equation}
\end{prop}

Before proving this proposition, let us show how this proposition implies Theorem \ref{th:mainNLS}.

\begin{proof}[Proof of Theorem \ref{th:mainNLS}]
Let $(\an)\in\R^N$. Denote $\sigma$ the permutation of $\unn$ which satisfies \[ c_{\sigma(1)}\leq \cdots\leq c_{\sigma(N)},\ \m{and}\ \sigma(i)<\sigma(j)\ \m{if}\ c_{\sigma(i)}=c_{\sigma(j)}\ \m{and}\ i<j. \]
\begin{enumerate}[(i)]
\item Consider $\phib_{A_{\sigma(1)}}$ the solution of \eqref{eq:NLS} given by Proposition \ref{th:princNLS} applied with $\phib$ given by Theorem \ref{th:cmmNLS}. Thus, there exists $t_0>0$ such that \[ \forall t\geq t_0,\quad \nh{\phib_{A_{\sigma(1)}}(t)-\phib(t)-A_{\sigma(1)}e^{-e_{\sigma(1)} t}Y_{\sigma(1)}^+(t)} \leq e^{-(e_{\sigma(1)}+\g)t}. \] Now, remark that $\phib_{A_{\sigma(1)}}$ is also a multi-soliton which satisfies \eqref{eq:cmmNLS}. Hence, we can apply Proposition \ref{th:princNLS} with $\phib_{A_{\sigma(1)}}$ instead of $\phib$, so that we obtain $\phib_{A_{\sigma(1)},A_{\sigma(2)}}$ such that \[ \forall t\geq t'_0,\quad \nh{\phib_{A_{\sigma(1)},A_{\sigma(2)}}(t) -\phib_{A_{\sigma(1)}}(t)-A_{\sigma(2)} e^{-e_{\sigma(2)}t}Y_{\sigma(2)}^+(t)} \leq e^{-(e_{\sigma(2)}+\g)t}. \] Similarly, for all $j\in\ient{2}{N}$, we construct by induction a solution $\phib_{A_{\sigma(1)},\ldots,A_{\sigma(j)}}$ such that \begin{equation} \label{eq:phiajNLS} \forall t\geq t_0,\quad \nh{\phib_{A_{\sigma(1)},\ldots,A_{\sigma(j)}}(t) -\phib_{A_{\sigma(1)},\ldots,A_{\sigma(j-1)}}(t)-A_{\sigma(j)}e^{-e_{\sigma(j)}t}Y_{\sigma(j)}^+(t)} \leq e^{-(e_{\sigma(j)}+\g)t}. \end{equation} Observe finally that $\phib_{\an} := \phib_{A_{\sigma(1)},\ldots,A_{\sigma(N)}}$ constructed by this way satisfies \eqref{eq:cmmNLS}.
\item Let $(A'_1,\ldots,A'_N)\in\R^N$ be such that $\phib_{A'_1,\ldots,A'_N}=\phib_{\an}$, and let us show that it implies $(A'_1,\ldots,A'_N)=(\an)$. In fact, we prove by induction on $j$ that $A_{\sigma(j)}=A'_{\sigma(j)}$ for all $j\in\unn$. For $j=1$, first note that, from the construction of $\phib_{\an}$, the hypothesis means $\phib_{A'_{\sigma(1)},\ldots,A'_{\sigma(N)}}= \phib_{A_{\sigma(1)},\ldots,A_{\sigma(N)}}$, and moreover \begin{align*} \phib_{A_{\sigma(1)},\ldots,A_{\sigma(N)}}(t) &= \phib_{A_{\sigma(1)},\ldots,A_{\sigma(N-1)}}(t) +A_{\sigma(N)}e^{-e_{\sigma(N)}t}Y_{\sigma(N)}^+(t) +z_{\sigma(N)}(t)\\ &= \cdots = \phib(t) +\sum_{k=1}^N A_{\sigma(k)}e^{-e_{\sigma(k)}t}Y_{\sigma(k)}^+(t) +\sum_{k=1}^N z_{\sigma(k)}(t), \end{align*} where $z_{\sigma(k)}$ satisfies $\nh{z_{\sigma(k)}(t)}\leq e^{-(e_{\sigma(k)}+\g)t}$ for $t\geq t_0$ and each $k\in\unn$. Similarly, we get \[ \phib_{A'_{\sigma(1)},\ldots,A'_{\sigma(N)}}(t) = \phib(t) +\sum_{k=1}^N A'_{\sigma(k)}e^{-e_{\sigma(k)}t}Y_{\sigma(k)}^+(t) +\sum_{k=1}^N \wt{z_{\sigma(k)}}(t), \] and so, by difference, we have \[ (A_{\sigma(1)}-A'_{\sigma(1)})e^{-e_{\sigma(1)}t}Y_{\sigma(1)}^+(t) +\sum_{k=2}^N (A_{\sigma(k)}-A'_{\sigma(k)})e^{-e_{\sigma(k)}t}Y_{\sigma(k)}^+(t) +\sum_{k=1}^N z_{\sigma(k)}(t)-\wt{z_{\sigma(k)}}(t) = 0. \] Now, if we multiply this equality by $Y_{\sigma(1)}^+(t)$, integrate, and take the imaginary part of it, we obtain, by Claim \ref{th:normalization} and \eqref{eq:interactNLS}, \[ |A_{\sigma(1)}-A'_{\sigma(1)}|e^{-e_{\sigma(1)}t}\leq Ce^{-(e_{\sigma(1)}+\g)t}, \] and so $A_{\sigma(1)}=A'_{\sigma(1)}$ by taking $t\to+\infty$. For the inductive step from $j-1$ to $j$, we write similarly \begin{align*} \phib_{A_{\sigma(1)},\ldots,A_{\sigma(N)}}(t) &= \phib_{A_{\sigma(1)},\ldots,A_{\sigma(j-1)}}(t) +\sum_{k=j}^N A_{\sigma(k)}e^{-e_{\sigma(k)}t}Y_{\sigma(k)}^+(t) +\sum_{k=j}^N z_{\sigma(k)}(t)\\ &= \phib_{A_{\sigma(1)},\ldots,A_{\sigma(j-1)}}(t) +\sum_{k=j}^N A'_{\sigma(k)}e^{-e_{\sigma(k)}t}Y_{\sigma(k)}^+(t) +\sum_{k=j}^N \wt{z_{\sigma(k)}}(t), \end{align*} and we finally obtain $A_{\sigma(j)} = A'_{\sigma(j)}$ as expected, by taking the difference of these two expressions, multiplying by $Y_{\sigma(j)}^+(t)$, integrating and taking the imaginary part of it. \qedhere
\end{enumerate}
\end{proof}

Now, the only purpose of the rest of the paper is to prove Proposition \ref{th:princNLS}. Let $j\in \unn$ and $A_j\in\R$, and denote $r_j(t,x)=A_je^{-e_jt}Y_j^+(t,x)=A_je^{-e_jt}Y_{c_j}^+(\lambda_j(t,x))e^{i\theta_j(t,x)}$. We want to construct a solution $u$ of \eqref{eq:NLS} such that \[ z(t,x)= u(t,x)-\phib(t,x)-r_j(t,x) \] satisfies $\nh{z(t)}\leq e^{-(e_j+\g)t}$ for $t\geq t_0$ with $t_0$ large enough.

\subsection{Equation of $z$} \label{sec:equationofz}

Since $u$ is a solution of \eqref{eq:NLS} and also $\phib$ is (and this fact is crucial for the whole proof), we get \[ i\partial_t z+\px^2 z+ {|\phib+r_j+z|}^{p-1}(\phib+r_j+z)-{|\phib|}^{p-1}\phib +A_je^{-e_jt}e^{i\theta_j}[\px^2 Y_{c_j}^+-c_jY_{c_j}^+ -ie_jY_{c_j}^+](\lambda_j)=0. \] But from Corollary \ref{th:spectrumcNLS}, we have \[ \L_{c_j} Y_{c_j}^+ = e_jY_{c_j}^+ = e_jY_{c_j,1}^+ +ie_jY_{c_j,2}^+ = -L_-Y_{c_j,2}^+ +iL_+ Y_{c_j,1}^+ \] where $Y_{c_j,1}^+=\re Y_{c_j}^+$ and $Y_{c_j,2}^+=\im Y_{c_j}^+$, and so \begin{equation} \label{eq:LYP} \px^2 Y_{c_j}^+ -c_jY_{c_j}^+ +iQ_{c_j}^{p-1}Y_{c_j,2}^+ +pQ_{c_j}^{p-1}Y_{c_j,1}^+ = ie_jY_{c_j}^+. \end{equation} Therefore, we get the following equation for $z$: \begin{equation} \label{eq:z1NLS} i\partial_t z +\px^2 z + {|\phib+r_j+z|}^{p-1}(\phib+r_j+z)-{|\phib|}^{p-1}\phib -A_je^{-e_jt}Q_{c_j}^{p-1}(\lambda_j) e^{i\theta_j} [pY_{c_j,1}^++iY_{c_j,2}^+](\lambda_j)=0. \end{equation} By developing the nonlinearity, we find \begin{multline*} {|\phib+r_j+z|}^{p-1}(\phib+r_j+z)-{|\phib|}^{p-1}\phib = {|\phib+r_j|}^{p-1}(\phib+r_j)-{|\phib|}^{p-1}\phib +\omega(z)\\ +(p-1){|\phib+r_j|}^{p-3}(\phib+r_j)\re((\Bar\phib+\Bar{r_j})z) +{|\phib+r_j|}^{p-1}z, \end{multline*} where $\omega(z)$ satisfies $|\omega(z)|\leq C{|z|}^2$ for $|z|\leq 1$. Hence, we can rewrite \eqref{eq:z1NLS} as \[ i\partial_t z+\px^2 z +(p-1){|\phib+r_j|}^{p-3}(\phib+r_j)\re((\Bar\phib+\Bar{r_j})z) +{|\phib+r_j|}^{p-1}z +\omega(z) = -\Omega, \] where \begin{equation} \label{eq:Omega} \Omega = {|\phib+r_j|}^{p-1}(\phib+r_j)-{|\phib|}^{p-1}\phib -A_je^{-e_jt}Q_{c_j}^{p-1}(\lambda_j)e^{i\theta_j} [pY_{c_j,1}^+ +iY_{c_j,2}^+](\lambda_j). \end{equation} Finally, the equation of $z$ can be written in the shorter form \begin{equation} \label{eq:z2NLS} i\partial_t z+\px^2 z +(p-1){|\phib|}^{p-3}\phib \re(\Bar\phib z) +{|\phib|}^{p-1}z +\omega_1\cdot z+\omega(z)=-\Omega, \end{equation} where $\omega_1$ satisfies $\nld{\omega_1(t)}\leq Ce^{-e_jt}$ for all $t\geq T_0$. We finally estimate the source term $\Omega$ in the following lemma, that we prove in Appendix \ref{app:Omega}.

\begin{lem} \label{th:Omegadecroit}
There exists $C>0$ such that, for all $t\geq T_0$, $\nh{\Omega(t)}\leq Ce^{-(e_j+4\g)t}$.
\end{lem}

\subsection{Compactness argument assuming uniform estimates}

To prove Proposition \ref{th:princNLS}, we follow the strategy of \cite{martel:NLS,martel:Nsolitons}. We first need some notation for our purpose.

\begin{notation}
\begin{enumerate}[(i)]
\item Denote $J=\{k\in\unn\ |\ c_k\leq c_j\}$, $K=\{k\in\unn\ |\ c_k>c_j\}$ and $k_0=\sharp K$.
\item $\R^{k_0}$ is equipped with the $\ell^2$ norm, simply denoted $\|\cdot\|$.
\item $\S_{\R^{k_0}}(r)$ denotes the sphere of radius $r$ in $\R^{k_0}$.
\item $B_{\mathcal{B}}(r)$ is the closed ball of the Banach space $\mathcal{B}$, centered at the origin and of radius $r\geq 0$.
\end{enumerate}
\end{notation}

Let $S_n\to+\infty$ be an increasing sequence of time, $\b_n={(b_{n,k})}_{k\in K} \in\R^{k_0}$ be a sequence of parameters to be determined, and let $u_n$ be the solution of \begin{equation} \label{unNLS} \begin{cases} i\partial_t u_n + \px^2 u_n + {|u_n|}^{p-1}u_n=0,\\ \displaystyle u_n(S_n) = \phib(S_n) +A_je^{-e_jS_n}Y_j^+(S_n)+\sum_{k\in K} b_{n,k}Y_k^+(S_n). \end{cases} \end{equation}

\begin{prop} \label{th:princbisNLS}
There exist $n_0\geq 0$ and $t_0>0$ (independent of $n$) such that the following holds. For each $n\geq n_0$, there exists $\b_n\in\R^{k_0}$ with $\| \b_n\| \leq 2e^{-(e_j+2\g)S_n}$, and such that the solution $u_n$ of \eqref{unNLS} is defined on the interval $[t_0,S_n]$, and satisfies \[ \forall t\in [t_0,S_n],\quad \nh{u_n(t)-\phib(t) -A_je^{-e_jt}Y_j^+(t)} \leq e^{-(e_j+\g)t}. \]
\end{prop}

Assuming this key proposition of uniform estimates, we can sketch the proof of Proposition \ref{th:princNLS}, relying on compactness arguments developed in \cite{martel:NLS,martel:Nsolitons}. The proof of Proposition \ref{th:princbisNLS} is postponed to the next section.

\begin{proof}[Sketch of the proof of Proposition \ref{th:princNLS} assuming Proposition \ref{th:princbisNLS}]
From Proposition \ref{th:princbisNLS}, there exists a sequence $u_n(t)$ of solutions to \eqref{eq:NLS}, defined on $[t_0,S_n]$, such that the following uniform estimates hold: \[ \forall n\geq n_0, \forall t\in [t_0,S_n],\quad \nh{u_n(t)-\phib(t)-A_je^{-e_jt}Y_j^+(t)}\leq e^{-(e_j+\g)t}. \] In particular, there exists $C_0>0$ such that $\nh{u_n(t_0)}\leq C_0$ for all $n\geq n_0$. Thus, there exists $u_0\in H^1(\R)$ such that $u_n(t_0)\cvf u_0$ in $H^1$ weak (after passing to a subsequence). Moreover, using the compactness result \cite[Lemma 2]{martel:NLS}, we can suppose that $u_n(t_0)\to u_0$ in $L^2$ strong, and so in $H^{s_p}$ strong by interpolation, where $0\leq s_p<1$ is an exponent for which local well-posedness and continuous dependence hold, according to a result of Cazenave and Weissler \cite{cazenaveweissler}. Now, consider $u$ solution of \[ \begin{cases} i\partial_t u +\px^2 u+{|u|}^{p-1}u=0,\\ u(t_0)=u_0. \end{cases} \] Fix $t\geq t_0$. For $n$ large enough, we have $S_n>t$, so $u_n(t)$ is defined and by continuous dependence of the solution of \eqref{eq:NLS} upon the initial data, we have $u_n(t)\to u(t)$ in $H^{s_p}$ strong. By the uniform $H^1$ bound, we also obtain $u_n(t)\cvf u(t)$ in $H^1$ weak. As \[ \nh{u_n(t)-\phib(t)-A_je^{-e_jt}Y_j^+(t)} \leq e^{-(e_j+\g)t}, \] we finally obtain, by weak convergence, $\nh{u(t)-\phib(t)-A_je^{-e_jt}Y_j^+(t)}\leq e^{-(e_j+\g)t}$. Thus, $u$ is a solution of \eqref{eq:NLS} which satisfies \eqref{eq:perturbNLS}.
\end{proof}

\subsection{Proof of Proposition \ref{th:princbisNLS}}

The proof proceeds in several steps. For the sake of simplicity, we will drop the index $n$  for the rest of this section (except for $S_n$). As Proposition \ref{th:princbisNLS} is proved for given $n$, this should not be a source of confusion. Hence, we will write $u$ for $u_n$, $z$ for $z_n$, $\b$ for $\b_n$, etc. We possibly drop the first terms of the sequence $S_n$, so that, for all $n$, $S_n$ is large enough for our purposes.

From \eqref{eq:z2NLS}, the equation satisfied by $z$ is \begin{equation} \begin{cases} \label{eq:z3NLS} i\partial_t z+\px^2 z+ (p-1){|\phib|}^{p-3}\phib\re(\Bar\phib z)+{|\phib|}^{p-1}z + \omega_1\cdot z +\omega(z) = -\Omega,\\ z(S_n) = \sum_{k\in K} b_k Y_k^+(S_n). \end{cases} \end{equation} Moreover, for all $k\in\unn$, we denote \[ \alpha_k^{\pm}(t) = \im \int \bar z(t)\cdot Y_k^{\pm}(t). \] In particular, we have \[ \alpha_k^{\pm}(S_n) = -\sum_{l\in K} b_l\im\int Y_{c_k}^{\mp}(\lambda_k(S_n))Y_{c_l}^+(\lambda_l(S_n)) e^{-i\theta_k(S_n)} e^{i\theta_l(S_n)}. \] Finally, we denote $\alpham(t)={(\alpha_k^-(t))}_{k\in K}$.

\subsubsection{Modulated final data}

\begin{lem} \label{th:finaldataNLS}
For $n\geq n_0$ large enough, the following holds. For all $\a\in\R^{k_0}$, there exists a unique $\b\in\R^{k_0}$ such that $\|\b\| \leq 2\|\a\|$ and $\alpham(S_n)=\a$.
\end{lem}

\begin{proof}
Consider the linear application \[ \begin{array}{rrcl} \Phi~: &\R^{k_0} &\to &\R^{k_0}\\ &\b={(b_l)}_{l\in K} &\mapsto &{(\alpha_k^-(S_n))}_{k\in K}. \end{array} \] If we denote $(\sigma_1,\ldots,\sigma_{k_0})$ the canonical basis of $\R^{k_0}$, then, by the normalization of Claim \ref{th:normalization} and the definition of $Y_c^+$ in Corollary \ref{th:spectrumcNLS}, we have, for all $k\in \ient{1}{k_0}$, \[ {(\Phi(\sigma_k))}_k = -\im\int \Carre{Y_{c_k}^+} = -\im\int \Carre{Y^+} = 1. \] Moreover, from \eqref{eq:interactNLS}, there exists $C_0>0$ independent of $n$ such that, for $l\neq k$, \[ |{(\Phi(\sigma_k))}_l|\leq \int \left| Y_{c_l}^+(\lambda_l(S_n))||Y_{c_k}^+(\lambda_k(S_n)) \right| \leq C_0e^{-\g S_n}. \] Thus, by taking $n_0$ large enough, we have $\Phi = \mathrm{Id} +A_n$ where $\|A_n\|\leq \frac{1}{2}$, so $\Phi$ is invertible and $\|\Phi^{-1}\|\leq 2$. Finally, for a given $\a\in \R^{k_0}$, it is enough to define $\b$ by $\b=\Phi^{-1}(\a)$ to conclude the proof of Lemma \ref{th:finaldataNLS}.
\end{proof}

\begin{claim} \label{SnNLS}
The following estimates at $S_n$ hold:
\begin{itemize}
\item $|\alpha_k^+(S_n)|\leq Ce^{-2\g S_n}\|\b\|$ for all $k\in\unn$, since $\im\int Y_{c_k}^- Y_{c_k}^+ = \im\int {|Y_{c_k}^+|}^2 = 0$.
\item $|\alpha_k^-(S_n)|\leq Ce^{-2\g S_n}\|\b\|$ for all $k\in J$.
\item $\nh{z(S_n)}\leq C\|\b\|$.
\end{itemize}
\end{claim}

\subsubsection{Equations on $\alpha_k^{\pm}$}

Let $t_0>0$ independent of $n$ to be determined later in the proof, $\a\in B_{\R^{k_0}}(e^{-(e_j+2\g)S_n})$ to be chosen, $\b$ be given by Lemma \ref{th:finaldataNLS} and $u$ be the corresponding solution of \eqref{unNLS}. We now define the maximal time interval $[\ta,S_n]$ on which suitable exponential estimates hold.

\begin{defi} \label{def:taNLS}
Let $\ta$ be the infimum of $T\geq t_0$ such that, for all $t\in [T,S_n]$, both following properties hold: \begin{equation} \label{eq:taNLS} e^{(e_j+\g)t}z(t) \in B_{H^1}(1) \quad \m{and} \quad e^{(e_j+2\g)t} \alpham(t) \in B_{\R^{k_0}}(1). \end{equation}
\end{defi}

Observe that Proposition \ref{th:princbisNLS} is proved if, for all $n$, we can find $\a$ such that $\ta = t_0$. The rest of the proof is devoted to prove the existence of such a value of $\a$.

\bigskip

First, we prove the following estimate on $\alpha_k^{\pm}$.

\begin{claim}
For all $k\in\unn$ and all $t\in [\ta,S_n]$, \begin{equation} \label{eq:alphaNLS} \left|\dt \alpha_k^{\pm}(t) \mp e_k\alpha_k^{\pm}(t) \right| \leq C_0e^{-4\g t}\nh{z(t)} +C_1\nh{z(t)}^2 +C_2e^{-(e_j+4\g)t}. \end{equation}
\end{claim}

\begin{proof}
Following Notation \ref{RjNLS}, we compute \begin{align*} \dt \alpha_k^{\pm}(t) &= -\dt \im\int \Bar{Y_k^{\pm}}(t)z(t) = -\dt \im\int Y_{c_k}^{\mp}(x-v_kt-x_k)e^{-i(\frac{1}{2}v_kx -\frac{1}{4}v_k^2t+c_kt+\g_k)}z(t)\\ &= -\im\int \left[ -v_k\px Y_{c_k}^{\mp}-i(c_k-\frac{1}{4}v_k^2)Y_{c_k}^{\mp}\right](x-v_kt-x_k) e^{-i(\frac{1}{2}v_kx -\frac{1}{4}v_k^2t+c_kt+\g_k)}z(t)\\ &\quad -\im\int Y_{c_k}^{\mp}(x-v_kt-x_k)e^{-i(\frac{1}{2}v_kx -\frac{1}{4}v_k^2t+c_kt+\g_k)}z_t. \end{align*} Moreover, using the equation of $z$ \eqref{eq:z3NLS} and an integration by parts, we find for the second term \begin{align*} -\im &\int Y_{c_k}^{\mp}(x-v_kt-x_k)e^{-i(\frac{1}{2}v_kx -\frac{1}{4}v_k^2t+c_kt+\g_k)}z_t\\ &= -\im\int Y_{c_k}^{\mp}(\lambda_k)e^{-i\theta_k}\times i\left[ \px^2 z+ (p-1){|\phib|}^{p-3}\phib\re(\Bar\phib z)+{|\phib|}^{p-1}z + \omega_1\cdot z +\omega(z) +\Omega\right] \displaybreak[0] \\ &= -\im\int ize^{-i\theta_k} \left[ \px^2 Y_{c_k}^{\mp}-iv_k\px Y_{c_k}^{\mp} -\frac{v_k^2}{4}Y_{c_k}^{\mp}\right](\lambda_k) \\ &\quad -\im\int iY_{c_k}^{\mp}(\lambda_k)e^{-i\theta_k}\left[ (p-1){|\phib|}^{p-3}\phib\re(\Bar\phib z)+{|\phib|}^{p-1}z \right]\\ &\quad -\im\int iY_{c_k}^{\mp}(\lambda_k)e^{-i\theta_k}\left[ \omega_1\cdot z+\omega(z)+\Omega\right].
\end{align*} Using the estimate $\nld{\omega_1(t)}\leq Ce^{-e_jt}$ and Lemma \ref{th:Omegadecroit}, we find for the last term \[ \left| -\im\int iY_{c_k}^{\mp}(\lambda_k)e^{-i\theta_k}\left[ \omega_1\cdot z+\omega(z)+\Omega\right] \right| \leq Ce^{-e_jt}\nh{z}+C\nh{z}^2 +Ce^{-(e_j+4\g)t}. \] From the definition of $\g$ \eqref{gammaNLS}, we deduce that \begin{align*} \dt \alpha_k^{\pm}(t) &= -\im\int ize^{-i\theta_k}\left[ \px^2 Y_{c_k}^{\mp}-c_kY_{c_k}^{\mp}\right] (\lambda_k) +O(e^{-4\g t}\nh{z}) +O(\nh{z}^2) +O(e^{-(e_j+4\g)t})\\ &\quad -\im\int iY_{c_k}^{\mp}(\lambda_k)e^{-i\theta_k}\left[ (p-1){|\phib|}^{p-3}\phib\re(\Bar\phib z)+{|\phib|}^{p-1}z \right]. \end{align*} Now, from \eqref{eq:LYP}, we find \[ -\im\int ize^{-i\theta_k}\left[ \px^2 Y_{c_k}^{\mp}-c_kY_{c_k}^{\mp}\right] (\lambda_k) = -\im\int ize^{-i\theta_k}\left[\mp ie_kY_{c_k}^{\mp}-iQ_{c_k}^{p-1}Y_{c_k,2}^{\mp} -pQ_{c_k}^{p-1}Y_{c_k,1}^{\mp}\right](\lambda_k), \] and, as in the proof of Lemma~\ref{th:Omegadecroit}, we also find \begin{multline*} -\im\int iY_{c_k}^{\mp}(\lambda_k)e^{-i\theta_k}\left[ (p-1){|\phib|}^{p-3}\phib\re(\Bar\phib z)+{|\phib|}^{p-1}z \right]\\ = -\im\int iY_{c_k}^{\mp}(\lambda_k)e^{-i\theta_k} \left[ (p-1){|R_k|}^{p-3}R_k\re(\Bar{R_k}z) +{|R_k|}^{p-1}z\right] + O(e^{-4\g t}\nh{z}). \end{multline*} Hence, we have \begin{align*} \dt \alpha_k^{\pm}(t) &= \pm\left( -\im\int ze^{-i\theta_k}Y_{c_k}^{\mp}(\lambda_k)\right) +\im\int ize^{-i\theta_k}\left[ iQ_{c_k}^{p-1}Y_{c_k,2}^{\mp} +pQ_{c_k}^{p-1}Y_{c_k,1}^{\mp}\right](\lambda_k)\\ &\quad -\im\int iY_{c_k}^{\mp}(\lambda_k)e^{-i\theta_k}\Big[ (p-1)Q_{c_k}^{p-2}(\lambda_k)e^{i\theta_k} \re[Q_{c_k}(\lambda_k)e^{-i\theta_k}z] +Q_{c_k}^{p-1}(\lambda_k)z \Big]\\ &\quad +O(e^{-4\g t}\nh{z}) +O(\nh{z}^2) +O(e^{-(e_j+4\g)t}). \end{align*} Finally, if we denote $z_1=\re(ze^{-i\theta_k})$ and $z_2=\im(ze^{-i\theta_k})$, we find \begin{align*} \dt \alpha_k^{\pm}(t) &= \pm e_k \alpha_k^{\pm}(t) +O(e^{-4\g t}\nh{z}) +O(\nh{z}^2) +O(e^{-(e_j+4\g)t})\\ &\quad +\re\int (z_1+iz_2)\left[ iQ_{c_k}^{p-1}(\lambda_k)Y_{c_k,2}^{\mp}(\lambda_k) +pQ_{c_k}^{p-1}(\lambda_k)Y_{c_k,1}^{\mp}(\lambda_k)\right]\\ &\quad -\re\int (p-1)Y_{c_k}^{\mp}(\lambda_k)Q_{c_k}^{p-1}(\lambda_k)z_1 -\re\int Y_{c_k}^{\mp}(\lambda_k)Q_{c_k}^{p-1}(\lambda_k)(z_1+iz_2)\\ &= \pm e_k \alpha_k^{\pm}(t) +O(e^{-4\g t}\nh{z}) +O(\nh{z}^2) +O(e^{-(e_j+4\g)t})\\ &\quad +p\int z_1Q_{c_k}^{p-1}(\lambda_k)Y_{c_k,1}^{\mp}(\lambda_k) -\int z_2Q_{c_k}^{p-1}(\lambda_k) Y_{c_k,2}^{\mp}(\lambda_k)\\ &\quad -(p-1)\int Y_{c_k,1}^{\mp}(\lambda_k)Q_{c_k}^{p-1}(\lambda_k)z_1 -\int Y_{c_k,1}^{\mp}(\lambda_k)Q_{c_k}^{p-1}(\lambda_k)z_1 +\int Y_{c_k,2}^{\mp}(\lambda_k)Q_{c_k}^{p-1}(\lambda_k)z_2\\ &= \pm e_k \alpha_k^{\pm}(t) +O(e^{-4\g t}\nh{z}) +O(\nh{z}^2) +O(e^{-(e_j+4\g)t}), \end{align*} since all other terms cancel.
\end{proof}

\subsubsection{Control of the stable directions}

We estimate here $\alpha_k^+(t)$ for all $k\in\unn$ and $t\in [\ta,S_n]$. From \eqref{eq:alphaNLS} and \eqref{eq:taNLS}, we have \[ \left| \dt \alpha_k^+(t) -e_k\alpha_k^+(t)\right| \leq C_0e^{-(e_j+5\g)t} +C_1e^{-2(e_j+\g)t} +C_2e^{-(e_j+4\g)t} \leq K_2e^{-(e_j+4\g)t}. \] Thus, $|{(e^{-e_ks}\alpha_k^+(s))}'|\leq K_2e^{-(e_j+e_k+4\g)s}$, and so, by integration on $[t,S_n]$, we get $|e^{-e_kS_n}\alpha_k^+(S_n)-e^{-e_kt}\alpha_k^+(t)| \leq K_2e^{-(e_j+e_k+4\g)t}$, which gives \[ |\alpha_k^+(t)|\leq e^{e_k(t-S_n)}|\alpha_k^+(S_n)|+K_2e^{-(e_j+4\g)t}. \] But from Claim \ref{SnNLS} and Lemma \ref{th:finaldataNLS}, we have \begin{align*} e^{e_k(t-S_n)}|\alpha_k^+(S_n)| \leq |\alpha_k^+(S_n)| &\leq Ce^{-2\g S_n}\|\b\|\\ &\leq Ce^{-2\g S_n} e^{-(e_j+2\g)S_n} \leq K_2e^{-(e_j+4\g)S_n} \leq K_2e^{-(e_j+4\g)t}, \end{align*} and so finally \begin{equation} \label{eq:alphap2NLS} \forall k\in\unn, \forall t\in[\ta,S_n],\quad |\alpha_k^+(t)|\leq K_2e^{-(e_j+4\g)t}. \end{equation}

\subsubsection{Control of the unstable directions for $k\in J$}

We estimate here $\alpha_k^-(t)$ for all $k\in J$ and $t\in [\ta,S_n]$. Note first that, as in the previous paragraph, we get, for all $k\in\unn$ and $t\in [\ta,S_n]$, \begin{equation} \label{eq:alpham2NLS} \left| \dt \alpha_k^-(t)+e_k\alpha_k^-(t)\right| \leq K_2e^{-(e_j+4\g)t}. \end{equation} Now suppose $k\in J$, which implies $e_k\leq e_j$. Since $|{(e^{e_ks}\alpha_k^-(s))}'|\leq K_2e^{(e_k-e_j-4\g)s}$, we obtain, by integration on $[t,S_n]$, \[ |\alpha_k^-(t)|\leq e^{e_k(S_n-t)}|\alpha_k^-(S_n)| +K_2e^{-(e_j+4\g)t}. \] But again from Claim \ref{SnNLS} and Lemma \ref{th:finaldataNLS}, we have \begin{align*} e^{e_k(S_n-t)}|\alpha_k^-(S_n)| &\leq K_2 e^{e_k(S_n-t)}e^{-2\g S_n}e^{-(e_j+2\g)S_n} = K_2 e^{e_k(S_n-t)}e^{-(e_j+4\g)S_n}\\ &\leq K_2e^{(S_n-t)(e_k-e_j)}e^{-e_jt}e^{-4\g S_n} \leq K_2e^{-(e_j+4\g)t}, \end{align*} and so finally \begin{equation} \label{eq:alphambisNLS} \forall k\in J, \forall t\in[\ta,S_n],\quad |\alpha_k^-(t)|\leq K_2e^{-(e_j+4\g)t}. \end{equation}

\subsubsection{Localized Weinstein's functional}

We follow here the same strategy as in \cite{martel:tsaiNLS,martel:NLS, martel:Nsolitons} to estimate the energy backwards. For this, we define the function $\psi$ by \[ \psi(x)=0 \m{ for } x\leq -1,\quad \psi(x)=1 \m{ for } x\geq 1,\quad \psi(x)= \frac{1}{c_0}\int_{-1}^x e^{-\frac{1}{1-y^2}}\,dy\quad \m{for } x\in (-1,1), \] where $c_0 = \int_{-1}^1 e^{-\frac{1}{1-y^2}}\,dy$. Hence, $\psi\in C^{\infty}(\R)$ is non-decreasing and $0\leq \psi\leq 1$. Moreover, we define, for all $k\in \ient{2}{N}$, $m_k(t)=\frac{1}{2}\left[ (v_k+v_{k-1})t+x_k+x_{k-1}\right]$, and \[ \psi_k(t,x)=\psi\left[ \frac{1}{\sqrt t}(x-m_k(t))\right],\quad \psi_1\equiv 1. \] Moreover, we set \begin{align*} h_1(t,x) &= \left( c_1+\frac{v_1^2}{4} \right) + \sum_{k=2}^{N} \left[ \left(c_k+\frac{v_k^2}{4}\right) -\left( c_{k-1}+\frac{v_{k-1}^2}{4}\right) \right] \psi_k(t,x),\\ h_2(t,x) &= v_1 +\sum_{k=2}^N (v_k-v_{k-1})\psi_k(t,x). \end{align*} Observe that the functions $h_1$ and $h_2$ take values close to $c_k+\frac{v_k^2}{4}$ and $v_k$ respectively, for $x$ close to $v_kt+x_k$, and have large variations only in regions far away from the solitons. To quantify these facts (see Lemma \ref{th:h}), we introduce the functions $\phi_k$, defined for $k\in \ient{1}{N-1}$ by \[ \phi_k = \psi_k -\psi_{k+1},\quad \phi_N = \psi_N. \] Hence, we have $\phi_k\geq 0$ and $\sum_{k=1}^N \phi_k\equiv 1$, and by an Abel's transform, we also have \[ h_1\equiv \sum_{k=1}^N \left( c_k+\frac{v_k^2}{4}\right)\phi_k\quad \m{and}\quad h_2\equiv \sum_{k=1}^N v_k\phi_k. \]

\begin{lem} \label{th:h}
\begin{enumerate}[(i)]
\item For all $k\in\unn$, $(|R_k|+|R_{kx}|)|\phi_k-1|\leq Ce^{-4\g t}e^{-\sqrt{\sigma_0}|x-v_kt|}$.
\item For all $k,l\in\unn$ such that $l\neq k$, $(|R_k|+|R_{kx}|)\phi_l \leq Ce^{-4\g t}e^{-\sqrt{\sigma_0}|x-v_kt|}$.
\item For all $k\in\unn$, $\nli{\phi_{kx}}+\nli{\phi_{kxx}}+ \nli{\phi_{kt}}\leq \frac{C}{\sqrt t}$.
\item One has $\nli{h_{1x}}+\nli{h_{2x}} +\nli{h_{1xx}}+\nli{h_{2xx}}+ \nli{h_{1t}}+\nli{h_{2t}}\leq \frac{C}{\sqrt t}$, and, for all $k\in\unn$, \begin{gather*} \left|h_1-\left(c_k+\frac{v_k^2}{4}\right)\right| (|R_k|+|R_{kx}|)\leq Ce^{-4\g t}e^{-\sqrt{\sigma_0}|x-v_kt|},\\ |h_2-v_k|(|R_k|+|R_{kx}|)\leq Ce^{-4\g t}e^{-\sqrt{\sigma_0}|x-v_kt|}. \end{gather*}
\end{enumerate}
\end{lem}

\begin{proof}
See Appendix \ref{app:Omega}.
\end{proof}

Now, we define a quantity related to the energy for $z$, by \begin{multline} \label{eq:defH} H(t) = \int {|\px z|}^2 -\frac{2}{p+1}\int {|\phib+r_j+z|}^{p+1}- {|\phib+r_j|}^{p+1} -(p+1){|\phib+r_j|}^{p-1}\re[(\Bar\phib+\Bar{r_j})z]\\ +\int h_1{|z|}^2 -\im\int h_2\bar z\px z. \end{multline} The following estimate of the variation of $H$ is the main new point of this paper, and as its proof is long and technical, it is postponed to Appendix \ref{app:dHdt}.

\begin{prop} \label{th:dHdtNLS}
For all $t\in [\ta,S_n]$, \[ \left| \frac{dH}{dt}(t)\right| \leq \frac{C_0}{\sqrt t}\nh{z(t)}^2 +C_1 e^{-(e_j+4\g)t}\nh{z(t)} +C_2\nh{z(t)}^3. \]
\end{prop}

We can now prove that, for all $t\in [\ta,S_n]$, \[ \H[z](t) := \int {|\px z|}^2 -{|R|}^{p-1}{|z|}^2 -(p-1)\Carre{\re(\Bar Rz)}{|R|}^{p-3} +h_1{|z|}^2 -\im h_2\bar z \px z \] satisfies \begin{equation} \label{eq:LhNLS} \H[z](t) \leq \frac{K_1}{\sqrt t}e^{-2(e_j+\g)t}. \end{equation} Indeed, from Proposition \ref{th:dHdtNLS} and estimates \eqref{eq:taNLS}, we deduce that, for all $s\in [t,S_n]$, \[ \left| \frac{dH}{ds}(s)\right| \leq \frac{C_0}{\sqrt s}e^{-2(e_j+\g)s} +C_1e^{-3\g s}e^{-2(e_j+\g)s} +C_2e^{-3(e_j+\g)s} \leq \frac{K_1}{\sqrt t}e^{-2(e_j+\g)s}. \] Thus, by integration on $[t,S_n]$, we obtain $|H(t)-H(S_n)| \leq \frac{K_1}{\sqrt t}e^{-2(e_j+\g)t}$, and so \[ H(t)\leq |H(S_n)|+\frac{K_1}{\sqrt t}e^{-2(e_j+\g)t}. \] But from Claim \ref{SnNLS} and Lemma \ref{th:finaldataNLS}, we have \begin{align*} |H(S_n)| &\leq C\nh{z(S_n)}^2 \leq C{\|\b\|}^2 \leq C{\|\a\|}^2\\ &\leq Ce^{-2(e_j+2\g)S_n} \leq Ce^{-2(e_j+2\g)t}, \end{align*} and so \[ \forall t\in [\ta,S_n],\quad H(t)\leq \frac{K_1}{\sqrt t}e^{-2(e_j+\g)t}. \] Finally, expanding ${|\phib+r_j+z|}^{p+1} = {\left[ {|\phib+r_j|}^2 +2\re[(\Bar\phib+\Bar{r_j})z] +{|z|}^2\right]}^{\frac{p+1}{2}}$, we find \begin{multline*} \left|{|\phib+r_j+z|}^{p+1}- {|\phib+r_j|}^{p+1} -(p+1)\re[(\Bar\phib+\Bar{r_j})z]{|\phib+r_j|}^{p-1} -\left(\frac{p+1}{2}\right){|z|}^2 {|\phib+r_j|}^{p-1} \right. \\ \left. -\frac{(p+1)(p-1)}{2}\Carre{\re[(\Bar\phib+\Bar{r_j})z]} {|\phib+r_j|}^{p-3} \right| \leq C{|z|}^3, \end{multline*} and so, from the definition of $H$ \eqref{eq:defH}, \[ \int {|\px z|}^2 -{|\phib+r_j|}^{p-1}{|z|}^2 -(p-1)\Carre{\re[(\Bar\phib+\Bar{r_j})z]}{|\phib+r_j|}^{p-3} +h_1{|z|}^2 -\im h_2\bar z \px z \leq \frac{K_1}{\sqrt t}e^{-2(e_j+\g)t}. \] Using \eqref{eq:cmmNLS}, we easily obtain \eqref{eq:LhNLS} by similar techniques used in the proof of Lemma \ref{th:Omegadecroit} in Appendix \ref{app:Omega} to replace $(\phib+r_j)$ by $R$ plus an exponentially small error term.

\subsubsection{Control of the directions of null energy} \label{sec:nullenergy}

Define $\displaystyle \zt(t) = z(t)+\sum_{k=1}^N \beta_k(t)iR_k(t) +\sum_{k=1}^N \g_k(t)\px Q_{c_k}(\lambda_k)e^{i\theta_k}$, where \[ \beta_k(t) = -\frac{\re\int iR_k\bar z}{\nld{Q_{c_k}}^2} = \frac{\im\int R_k\bar z}{\nld{Q_{c_k}}^2}\quad \m{and}\quad \g_k(t) = -\frac{\re\int \px Q_{c_k}(\lambda_k)e^{i\theta_k}\bar z}{\nld{\px Q_{c_k}}^2}. \] First, note that there exist $C_1,C_2>0$ such that \begin{equation} \label{eq:zztNLS} C_1\nh{z}\leq \nh{\zt}+\sum_{k=1}^N (|\beta_k| +|\g_k|) \leq C_2\nh{z}. \end{equation} Moreover, by this choice of parameters, we have, for all $k\in\unn$, \begin{equation} \label{eq:RkxNLS} \left| \re\int -i\Bar{R_k}\zt \right| \leq Ce^{-\g t}\nh{z},\quad \left| \re\int \px Q_{c_k}(\lambda_k)e^{i\theta_k}\Bar{\zt}\right| \leq Ce^{-\g t}\nh{z}. \end{equation} Indeed, by \eqref{eq:interactNLS}, we have \begin{align*} \re\int -i\Bar{R_k}\zt &= \im\int \Bar{R_k}\left[ z(t)+\sum_{l=1}^N \beta_l(t)iR_l(t) +\sum_{l=1}^N \g_l(t)\px Q_{c_l}(\lambda_l)e^{i\theta_l}\right] \\ &= \im\int \Bar{R_k}z +\beta_k(t)\re\int {|R_k|}^2 +\g_k(t)\im\int Q_{c_k}\px Q_{c_k} +O(e^{-\g t}\nh{z})\\ &= \im\int \Bar{R_k}z +\im\int R_k\bar z + O(e^{-\g t}\nh{z}) = O(e^{-\g t}\nh{z}), \end{align*} and similarly, \begin{align*} &\re\int \px Q_{c_k}(\lambda_k)e^{i\theta_k}\Bar{\zt}\\ &= \re\int \px Q_{c_k}(\lambda_k)e^{i\theta_k}\bar z +\beta_k(t)\im\int Q_{c_k}\px Q_{c_k} +\g_k(t)\re\int {|\px Q_{c_k}|}^2 +O(e^{-\g t}\nh{z})\\ &= \re\int \px Q_{c_k}(\lambda_k)e^{i\theta_k}\bar z -\re\int \px Q_{c_k}(\lambda_k)e^{i\theta_k}\bar z + O(e^{-\g t}\nh{z}) = O(e^{-\g t}\nh{z}). \end{align*}

Now, we compare the functionals $\H[\zt]$ and $\H[z]$ in the following lemma, that we prove in Appendix \ref{app:Omega}.

\begin{lem} \label{th:lienHH}
For all $t\in [\ta,S_n]$, one has \[ \H[\zt](t)\leq \H[z](t) +\frac{C}{\sqrt t}\nh{z}^2. \]
\end{lem}

By \eqref{eq:LhNLS} and \eqref{eq:taNLS}, we deduce that \begin{equation} \label{eq:LhtNLS} \forall t\in [\ta,S_n],\quad  \H[\zt](t) \leq \frac{K_1}{\sqrt t}e^{-2(e_j+\g)t}. \end{equation} Now, from the property of coercivity (ii) in Proposition \ref{th:coer}, and by the definitions of $h_1$ and $h_2$, we obtain, by simple localization arguments (see \cite[Appendix B]{martel:tsaiNLS} for details), that there exists $\kappa_1>0$ such that \begin{multline} \label{eq:Hzt} \H[\zt](t) \geq \frac{1}{\kappa_1}\nh{\zt}^2 -\kappa_1\sum_{k=1}^N \left[ \Carre{-\im\int \zt \Bar{Y_k^+}} + \Carre{-\im\int \zt \Bar{Y_k^-}} \right. \\ \left. + \Carre{\re\int \zt (-i\Bar{R_k})} + \Carre{\re \int \zt \px Q_{c_k}(\lambda_k)e^{-i\theta_k}}\right].  \end{multline} To justify heuristically this inequality, we compute, for $k\in\unn$, the localized version $\H_k[z]$ of $\H[z]$ (it would be the same for $\zt$), defined by \[ \H_k[z] = \int {|\px z|}^2 -{|R_k|}^{p-1}{|z|}^2 -(p-1)\Carre{\re(\Bar{R_k}z)} {|R_k|}^{p-3} +\left(c_k+\frac{v_k^2}{4}\right){|z|}^2 -v_k\im \bar z\px z. \] In fact, if we denote $[e^{-i\theta_k}z](\cdot+v_kt+x_k)=z_1+iz_2$, \emph{i.e.} $z=e^{i\theta_k}(z_1+iz_2)(\lambda_k)$, then we have $\px z = \frac{iv_k}{2}e^{i\theta_k}(z_1+iz_2)(\lambda_k) +e^{i\theta_k}(\px z_1+i\px z_2)(\lambda_k)$, and so, by (ii) of Proposition \ref{th:coer}, \begin{align*} \H_k[z] &= \int \Carre{-\frac{v_k}{2}z_2+\px z_1}(\lambda_k) +\int \Carre{\frac{v_k}{2}z_1+\px z_2}(\lambda_k)\\ &\quad -\int Q_{c_k}^{p-1}(\lambda_k)(z_1^2+z_2^2)(\lambda_k) -(p-1)\int Q_{c_k}^{p-1}(\lambda_k)z_1^2(\lambda_k)\\ &\quad +\int \left(c_k+\frac{v_k^2}{4}\right)(z_1^2+z_2^2)(\lambda_k) -v_k\int \left(\frac{v_k}{2}z_1^2 +z_1\px z_2 +\frac{v_k}{2}z_2^2-z_2\px z_1\right)(\lambda_k) \displaybreak[0] \\ &= \int {(\px z_1)}^2 +c_kz_1^2 -pQ_{c_k}^{p-1}z_1^2 +\int {(\px z_2)}^2 +c_kz_2^2 -Q_{c_k}^{p-1}z_2^2 = (L_{c_k+}z_1,z_1) +(L_{c_k-}z_2,z_2)\\ &\geq \frac{1}{\kappa_0}\nh{z}^2 -\kappa_0\left[ \Carre{\int \px Q_{c_k}z_1}+ \Carre{\int Q_{c_k}z_2} +\Carre{\im\int Y_k^+\bar z}+ \Carre{\im\int Y_k^- \bar z}\right].
\end{align*}

Now, we return to \eqref{eq:Hzt}, and we estimate each term of the sum, for all $k\in\unn$ and $t\in [\ta,S_n]$. First, by~\eqref{eq:RkxNLS}, we have \[ \Carre{\re\int \zt (-i\Bar{R_k})} + \Carre{\re \int \zt \px Q_{c_k}(\lambda_k)e^{-i\theta_k}} \leq Ce^{-2\g t}\nh{z}^2 \leq Ce^{-2\g t}e^{-2(e_j+\g)t}. \] Second, denoting $Y_1=\re Y^+$ and $Y_2=\im Y^+$ again, we have \begin{align*} -\im\int \Bar{Y_k^+}(t)\zt(t) &= \alpha_k^+(t)-\beta_k(t)\re\int Q_{c_k}(\lambda_k)(Y_{c_k,1}^+-iY_{c_k,2}^+)(\lambda_k)\\ &\quad -\g_k(t)\im\int \px Q_{c_k}(\lambda_k)(Y_{c_k,1}^+-iY_{c_k,2}^+)(\lambda_k) +O(e^{-\g t}\nh{z})\\ &= \alpha_k^+(t) -C\beta_k(t)\int QY_1 +C\g_k(t)\int\px QY_2 +O(e^{-\g t}\nh{z}). \end{align*} But by definition of $Y^+$, we recall that $L_+ Y_1=e_0Y_2$ and $L_- Y_2 = -e_0Y_1$, and so \begin{align*} -\im\int \Bar{Y_k^+}(t)\zt(t) &= \alpha_k^+(t) +\frac{C\beta_k(t)}{e_0}\int Q(L_-Y_2) +\frac{C\g_k(t)}{e_0}\int \px Q(L_+Y_1) + O(e^{-\g t}\nh{z})\\ &= \alpha_k^+(t) +C'\beta_k(t)\int (L_-Q)Y_2 +C'\g_k(t)\int L_+(\px Q)Y_1 + O(e^{-\g t}\nh{z})\\ &= \alpha_k^+(t) + O(e^{-\g t}\nh{z}), \end{align*} since $L_{\pm}$ are self-adjoint, and moreover, $L_-Q=0$ and $L_+(\px Q)=0$ by Proposition \ref{th:spectrumNLS}. Hence, by \eqref{eq:alphap2NLS}, we find, for all $k\in\unn$, \[ \Carre{-\im\int \zt \Bar{Y_k^+}} \leq 2{(\alpha_k^+)}^2 +Ce^{-2\g t}\nh{z}^2 \leq Ce^{-2(e_j+4\g)t} +Ce^{-2\g t}e^{-2(e_j+\g)t} \leq Ce^{-2\g t}e^{-2(e_j+\g)t}. \] Completely similarly, we find, for all $k\in\unn$, \[ \Carre{-\im\int \zt \Bar{Y_k^-}} \leq 2{(\alpha_k^-)}^2 +Ce^{-2\g t}\nh{z}^2\leq Ce^{-2\g t}e^{-2(e_j+\g)t}, \] using \eqref{eq:alphambisNLS} for $k\in J$, and \eqref{eq:taNLS} for $k\in K$.

Finally, gathering all estimates from \eqref{eq:LhtNLS}, we have proved that there exists $\wt{K_0}>0$ such that, for all $t\in [\ta,S_n]$, \[ \nh{\zt(t)}\leq \frac{\wt{K_0}}{t^{1/4}}e^{-(e_j+\g)t}. \] We want now to prove the same estimate for $z$, and so we have to control the parameters $\beta_k(t)$ and $\g_k(t)$ introduced above.

\subsubsection{Improvement of the decay of $z$}

\begin{lem} \label{th:zdecroitNLS}
There exists $K_0>0$ such that, for all $t\in [\ta,S_n]$, \[ \nh{z(t)}\leq \frac{K_0}{t^{1/4}}e^{-(e_j+\g)t}. \]
\end{lem}

\begin{proof}
By \eqref{eq:zztNLS}, it is enough to prove this estimate for $|\beta_k(t)|+|\g_k(t)|$ with $k\in\unn$ fixed. To do this, write first the equation of $\zt$, from the equation of $z$ \eqref{eq:z2NLS},
\begin{align*} &i\partial_t \zt +\px^2\zt +(p-1){|\phib|}^{p-3}\phib\re(\Bar\phib\zt)+{|\phib|}^{p-1}\zt\\ &= i\partial_t z-\sum \beta'_lR_l -\sum \beta_l\left[ -v_l\px Q_{c_l}+i\left(c_l-\frac{v_l^2}{4}\right) Q_{c_l}\right](\lambda_l)e^{i\theta_l} +i\sum \g'_l\px Q_{c_l}(\lambda_l)e^{i\theta_l}\\ &\quad +i\sum \g_l\left[ -v_l\px^2 Q_{c_l}+i\left(c_l-\frac{v_l^2}{4}\right)\px Q_{c_l}\right](\lambda_l)e^{i\theta_l} +\px^2 z\\ &\quad +i\sum \beta_l\left[ \px^2 Q_{c_l}+iv_l\px Q_{c_l}-\frac{v_l^2}{4}Q_{c_l}\right](\lambda_l)e^{i\theta_l} +\sum \g_l\left[ \px^3 Q_{c_l}+iv_l\px^2Q_{c_l}-\frac{v_l^2}{4}\px Q_{c_l}\right](\lambda_l)e^{i\theta_l}\\ &\quad +(p-1){|\phib|}^{p-3}\phib\re(\Bar\phib z) +(p-1){|\phib|}^{p-3}\phib\sum \beta_l\re(i\Bar\phib R_l) +{|\phib|}^{p-1}z\\ &\quad +(p-1){|\phib|}^{p-3}\phib\sum \g_l\re(\Bar\phib\px Q_{c_l}(\lambda_l)e^{i\theta_l}) +\sum \beta_l i{|\phib|}^{p-1}R_l +\sum \g_l {|\phib|}^{p-1}\px Q_{c_l}(\lambda_l)e^{i\theta_l}, \end{align*} and so, since $\px^2 Q_{c_l}+Q_{c_l}^p = c_lQ_{c_l}$, we find \begin{align*} &i\partial_t \zt +\px^2\zt +(p-1){|\phib|}^{p-3}\phib\re(\Bar\phib\zt)+{|\phib|}^{p-1}\zt\\ &= -\omega_1\cdot z-\omega(z)-\Omega -\sum \beta'_lR_l +i\sum \g'_l\px Q_{c_l}(\lambda_l)e^{i\theta_l}\\ &\quad -i\sum \beta_l Q_{c_l}^p(\lambda_l)e^{i\theta_l} -p\sum \g_l\px Q_{c_l}(\lambda_l)Q_{c_l}^{p-1}(\lambda_l)e^{i\theta_l}\\ &\quad -(p-1)\sum \beta_l {|\phib|}^{p-3}\phib\im(\Bar\phib R_l) +(p-1)\sum \g_l {|\phib|}^{p-3}\phib\re(\Bar\phib \px Q_{c_l}(\lambda_l)e^{i\theta_l})\\ &\quad +i\sum \beta_l {|\phib|}^{p-1}Q_{c_l}(\lambda_l)e^{i\theta_l} +\sum \g_l {|\phib|}^{p-1}\px Q_{c_l}(\lambda_l)e^{i\theta_l}\\ &=-\omega_1\cdot z-\omega(z)-\Omega -\sum \beta'_lR_l +i\sum \g'_l\px Q_{c_l}(\lambda_l)e^{i\theta_l} -(p-1)\sum \beta_l {|\phib|}^{p-3}\phib\im(\Bar\phib R_l)\\ &\quad +i\sum\beta_le^{i\theta_l}Q_{c_l}(\lambda_l) [{|\phib|}^{p-1}-Q_{c_l}^{p-1}(\lambda_l)]\\ &\quad + \sum \g_l \left[ {|\phib|}^{p-1}\px Q_{c_l}(\lambda_l)e^{i\theta_l} +(p-1){|\phib|}^{p-3}\phib\re(\Bar\phib \px Q_{c_l}(\lambda_l)e^{i\theta_l}) -p\px Q_{c_l}(\lambda_l)Q_{c_l}^{p-1}(\lambda_l)e^{i\theta_l} \right]. \end{align*}Then, multiply this equation by $\Bar{R_k}$, integrate, and take the real part of it, so that we obtain, by~\eqref{eq:interactNLS}, \eqref{eq:cmmNLS} and Lemma \ref{th:Omegadecroit}, \begin{multline*} -\im\int \partial_t\zt \Bar{R_k} +O(\nld{\zt}) = O(e^{-e_jt}\nh{z})+O(\nh{z}^2) +O(e^{-(e_j+4\g)t}) -C\beta'_k\\ +\sum_{l\neq k} (\beta'_l+\g'_l)O(e^{-\g t}) +\sum \beta_lO(e^{-\g t}) +\sum \g_lO(e^{-\g t}). \end{multline*} In other words, we have, by \eqref{eq:zztNLS} and \eqref{eq:taNLS}, \[ |\beta'_k|\leq C\left| \im\int \partial_t\zt \Bar{R_k} \right| + Ce^{-\g t}\sum_{l\neq k} (|\beta'_l|+|\g'_l|) +\frac{C}{t^{1/4}}e^{-(e_j+\g)t}. \] Moreover, from \[ \im\int\zt \Bar{R_k} = \sum_{l\neq k}\beta_l\im\int iR_l\Bar{R_k}+\sum_{l\neq k}\g_l \im\int \px Q_{c_l}(\lambda_l)e^{i\theta_l}\Bar{R_k}, \] we deduce that \begin{align*} \dt\im\int \zt \Bar{R_k} &= \sum_{l\neq k} (\beta'_l+\g'_l)O(e^{-\g t}) +\sum_{l\neq k} (\beta_l+\g_l)O(e^{-\g t})\\ &= \im\int \partial_t\zt \Bar{R_k} +\im\int \zt \partial_t\Bar{R_k}, \end{align*} and so, as $\partial_t R_k = -v_k\px R_k+i\left(c_k+\frac{v_k^2}{4}\right)R_k$, \[  \left| \im\int \partial_t\zt \Bar{R_k} \right| \leq C\nh{\zt} +Ce^{-\g t}\sum_{l\neq k} (|\beta'_l|+|\g'_l|) +Ce^{-\g t}\sum_{l\neq k} (|\beta_l|+|\g_l|). \] Gathering previous estimates, we find \[ |\beta'_k|\leq Ce^{-\g t}\sum_{l\neq k} (|\beta'_l|+|\g'_l|) +\frac{C}{t^{1/4}}e^{-(e_j+\g)t}. \] Completely similarly, if we multiply the equation on $\zt$ by $\px Q_{c_k}(\lambda_k)e^{-i\theta_k}$, integrate and take the imaginary part of it, we find \[ |\g'_k|\leq Ce^{-\g t}\sum_{l\neq k} (|\beta'_l|+|\g'_l|) +\frac{C}{t^{1/4}}e^{-(e_j+\g)t}. \]

Hence, we have proved that there exist $C_3,C_4>0$ such that, for all $t\in [\ta,S_n]$, \[ |\beta'_k| +|\g'_k| \leq C_3e^{-\g t}\sum_{l\neq k} (|\beta'_l|+|\g'_l|) +\frac{C_4}{t^{1/4}}e^{-(e_j+\g)t}. \] Finally, if we choose $t_0$ large enough so that $C_3e^{-\g t_0}\leq \frac{1}{N}$, we obtain, for all $s\in [t,S_n]$, with $t\in [\ta,S_n]$, \[ |\beta'_k(s)| +|\g'_k(s)| \leq \frac{C}{t^{1/4}}e^{-(e_j+\g)s}. \] By integration on $[t,S_n]$, we get $|\beta_k(t)| +|\g_k(t)| \leq |\beta_k(S_n)|+|\g_k(S_n)| + \frac{C}{t^{1/4}}e^{-(e_j+\g)t}$. But from Claim \ref{SnNLS}, Lemma \ref{th:finaldataNLS} and \eqref{eq:zztNLS}, we have \[ |\beta_k(S_n)|+|\g_k(S_n)| \leq C\nh{z(S_n)} \leq C\|\b\| \leq C\|\a\| \leq Ce^{-(e_j+2\g)S_n} \leq Ce^{-(e_j+2\g)t}, \] and so finally, \[ \forall t\in [\ta,S_n],\quad |\beta_k(t)|+ |\g_k(t)| \leq \frac{C}{t^{1/4}}e^{-(e_j+\g)t}. \qedhere \]
\end{proof}

\subsubsection{Control of the unstable directions for $k\in K$ by a topological argument}

Lemma \ref{th:zdecroitNLS} being proved, we choose $t_0$ large enough so that $\frac{K_0}{t_0^{1/4}}\leq \frac{1}{2}$. Therefore, we have \[ \forall t\in [\ta,S_n],\quad \nh{z(t)}\leq \frac{1}{2}e^{-(e_j+\g)t}. \] We can now prove the following final lemma, which concludes the proof of Proposition~\ref{th:princbisNLS}. Note that its proof is very similar to the one in \cite{combet:multisoliton}, by the common choice of notation, but it is reproduced here for the reader's convenience.

\begin{lem}
For $t_0$ large enough, there exists $\a\in B_{\R^{k_0}}(e^{-(e_j+2\g)S_n})$ such that $\ta=t_0$.
\end{lem}

\begin{proof}
For the sake of contradiction, suppose that, for all $\a\in B_{\R^{k_0}}(e^{-(e_j+2\g)S_n})$, $\ta>t_0$. As $e^{(e_j+\g)\ta}z(\ta)\in B_{H^1}(1/2)$, then, by definition of $\ta$ and continuity of the flow, we have \begin{equation} \label{eq:alphamtaNLS} e^{(e_j+2\g)\ta}\alpham(\ta) \in\S_{\R^{k_0}}(1). \end{equation} Now, let $T\in [t_0,\ta]$ be close enough to $\ta$ such that $z$ is defined on $[T,S_n]$, and by continuity, \[ \forall t\in [T,S_n],\quad \nh{z(t)}\leq e^{-(e_j+\g)t}. \] We can now consider, for $t\in [T,S_n]$, \[ \N(t) = \N(\alpham(t)) = \| e^{(e_j+2\g)t}\alpham(t)\|^2. \] To calculate $\N'$, we start from estimate \eqref{eq:alpham2NLS}: \[ \forall k\in K, \forall t\in [T,S_n],\quad \left| \dt \alpha_k^-(t) +e_k\alpha_k^-(t)\right| \leq K'_2e^{-(e_j+4\g)t}. \] Multiplying by $|\alpha_k^-(t)|$, we obtain \[ \left| \alpha_k^-(t)\dt\alpha_k^-(t) +e_k{\alpha_k^-(t)}^2 \right| \leq K'_2e^{-(e_j+4\g)t}|\alpha_k^-(t)|, \] and thus \[ 2\alpha_k^-(t)\dt\alpha_k^-(t) +2\emin {\alpha_k^-(t)}^2 \leq 2\alpha_k^-(t)\dt\alpha_k^-(t) +2e_k{\alpha_k^-(t)}^2 \leq K'_2e^{-(e_j+4\g)t}|\alpha_k^-(t)|, \] where $\emin = \min\{ e_k\ ;\ k\in K\}$. By summing on $k\in K$, we get \[ {(\|\alpham(t)\|^2)}' +2\emin \|\alpham(t)\|^2 \leq K_3e^{-(e_j+4\g)t}\|\alpham(t)\|. \]

Therefore, we can estimate \begin{align*} \N'(t) &= {(e^{2(e_j+2\g)t}\|\alpham(t)\|^2)}' = e^{2(e_j+2\g)t}\left[ 2(e_j+2\g)\|\alpham(t)\|^2 + {(\|\alpham(t)\|^2)}' \right]\\ &\leq e^{2(e_j+2\g)t} \left[ 2(e_j+2\g)\|\alpham(t)\|^2 -2\emin \|\alpham(t)\|^2 +K_3e^{-(e_j+4\g)t}\|\alpham(t)\| \right]. \end{align*} Hence, we have, for all $t\in [T,S_n]$, \[ \N'(t) \leq -\theta\cdot\N(t) +K_3e^{e_jt}\|\alpham(t)\|, \] where $\theta=2(\emin-e_j-2\g)>0$ by the definitions of $\g$ \eqref{gammaNLS} and of the set $K$. In particular, for all $\tau\in [T,S_n]$ satisfying $\N(\tau)=1$, we have \[ \N'(\tau) \leq -\theta +K_3e^{e_j\tau}\|\alpham(\tau)\| = -\theta +K_3e^{e_j\tau}e^{-(e_j+2\g)\tau} = -\theta +K_3e^{-2\g\tau} \leq -\theta +K_3e^{-2\g t_0}. \] Now, we definitely fix $t_0$ large enough so that $K_3e^{-2\g t_0}\leq \frac{\theta}{2}$, and so, for all $\tau\in [T,S_n]$ such that $\N(\tau)=1$, we have \begin{equation} \label{eq:NprimeNLS} \N'(\tau) \leq -\frac{\theta}{2}. \end{equation} In particular, by \eqref{eq:alphamtaNLS}, we have $\N'(\ta) \leq -\frac{\theta}{2}$.

\begin{description}
\item[First consequence:] $\a\mapsto \ta$ is continuous. Indeed, let $\eps>0$. Then there exists $\delta>0$ such that $\N(\ta-\eps)>1+\delta$ and $\N(\ta+\eps)<1-\delta$. Moreover, by definition of $\ta$ and \eqref{eq:NprimeNLS}, there can not exist $\tau\in [\ta+\eps,S_n]$ such that $\N(\tau)=1$, and so by choosing $\delta$ small enough, we have, for all $t\in [\ta+\eps,S_n]$, $\N(t)<1-\delta$. But from continuity of the flow, there exists $\eta>0$ such that, for all $\widetilde{\mathfrak{a}}^-$ satisfying $\|\widetilde{\mathfrak{a}}^- -\a\|\leq \eta$, we have \[ \forall t\in [\ta-\eps,S_n],\quad |\N(\widetilde{\boldsymbol{\alpha}}^-(t)) -\N(\alpham(t))| \leq \delta/2. \] We finally deduce that $\ta-\eps \leq T(\widetilde{\mathfrak{a}}^-)\leq \ta+\eps$, as expected.
\item[Second consequence:] We can define the map \[ \begin{array}{rrcl} \M~: &B_{\R^{k_0}}(e^{-(e_j+2\g)S_n}) &\to &\S_{\R^{k_0}}(e^{-(e_j+2\g)S_n})\\ &\a &\mapsto &e^{-(e_j+2\g)(S_n-\ta)}\alpham(\ta). \end{array} \] Note that $\M$ is continuous by the previous point. Moreover, let $\a\in\S_{\R^{k_0}}(e^{-(e_j+2\g)S_n})$. As $\N'(S_n)\leq -\frac{\theta}{2}$ by \eqref{eq:NprimeNLS}, we deduce by definition of $\ta$ that $\ta=S_n$, and so $\M(\a)=\a$. In other words, $\M$ restricted to $\S_{\R^{k_0}}(e^{-(e_j+2\g)S_n})$ is the identity. But the existence of such a map $\M$ contradicts Brouwer's fixed point theorem.
\end{description}

In conclusion, there exists $\a\in B_{\R^{k_0}}(e^{-(e_j+2\g)S_n})$ such that $\ta=t_0$.
\end{proof}

\appendix

\section{Appendix} \label{app:Omega}

\begin{proof}[Proof of Lemma \ref{th:Omegadecroit}]
First, we calculate \begin{align*} {|R_j|}^{p-1}r_j +(p-1){|R_j|}^{p-3}R_j\re(\Bar{R_j}r_j) &= A_je^{-e_jt}Q_{c_j}^{p-1}(\lambda_j)[Y_{c_j,1}^++iY_{c_j,2}^+](\lambda_j)e^{i\theta_j} \\ &+(p-1)Q_{c_j}^{p-2}(\lambda_j)e^{i\theta_j}\re[A_je^{-e_jt}Q_{c_j}(Y_{c_j,1}^++iY_{c_j,2}^+)](\lambda_j)\\ &= A_je^{-e_jt}Q_{c_j}^{p-1}(\lambda_j)e^{i\theta_j}[Y_{c_j,1}^++iY_{c_j,2}^+ +(p-1)Y_{c_j,1}^+](\lambda_j)\\ &= A_je^{-e_jt}Q_{c_j}^{p-1}(\lambda_j)e^{i\theta_j}[pY_{c_j,1}^++iY_{c_j,2}^+](\lambda_j). \end{align*}

Hence, from the expression of $\Omega$ \eqref{eq:Omega}, it can be written \[ \Omega = {|\phib+r_j|}^{p-1}(\phib+r_j)-{|\phib|}^{p-1}\phib -{|R_j|}^{p-1}r_j -(p-1){|R_j|}^{p-3}R_j\re(\Bar{R_j}r_j).\] We can now estimate $\nh{\Omega}$, and we estimate $\nld{\px \Omega}$ for example, the term $\nld{\Omega}$ being similar and easier. To do this, we write \begin{align*} \Omega_x &= (p-1)\re[(\phib_x+r_{jx})(\Bar\phib+\Bar{r_j})]{|\phib+r_j|}^{p-3}(\phib+r_j) +{|\phib+r_j|}^{p-1}(\phib_x+r_{jx})\\ &\quad -(p-1)\re(\phib_x\Bar\phib){|\phib|}^{p-3}\phib -{|\phib|}^{p-1}\phib_x -(p-1)\re(R_{jx}\Bar{R_j}){|R_j|}^{p-3}r_j -{|R_j|}^{p-1}r_{jx}\\ &\quad -(p-1)(p-3)\re(R_{jx}\Bar{R_j}){|R_j|}^{p-5}R_j\re(\Bar{R_j}r_j) -(p-1){|R_j|}^{p-3}R_{jx}\re(\Bar{R_j}r_j)\\ &\quad -(p-1){|R_j|}^{p-3}R_j\re(\Bar{R_{jx}}r_j) -(p-1){|R_j|}^{p-3}R_j\re(\Bar{R_j}r_{jx})\\ &= (p-1)\re(\phib_x\Bar\phib)\left[ {|\phib+r_j|}^{p-3}(\phib+r_j)-{|\phib|}^{p-3}\phib -(p-3)\phib \re(\Bar\phib r_j){|\phib|}^{p-5} -{|\phib|}^{p-3}r_j\right]\\ &\quad +(p-1)(p-3)\left[ \re(\phib_x\Bar\phib) \re(\Bar\phib r_j){|\phib|}^{p-5}\phib -\re(R_{jx}\Bar{R_j})\re(\Bar{R_j}r_j) {|R_j|}^{p-5}R_j\right]\\ &\quad +(p-1)r_j\left[ \re(\phib_x\Bar\phib){|\phib|}^{p-3} -\re(R_{jx}\Bar{R_j}){|R_j|}^{p-3}\right]\\ &\quad +(p-1)\left[ \re(\phib_x\Bar{r_j}){|\phib+r_j|}^{p-3}(\phib+r_j) -\re(\Bar{R_{jx}}r_j){|R_j|}^{p-3}R_j \right]\\ &\quad +(p-1)\left[ \re(r_{jx}\Bar\phib){|\phib+r_j|}^{p-3}(\phib+r_j) -\re(r_{jx}\Bar{R_j}){|R_j|}^{p-3}R_j\right]\\ &\quad +(p-1)\re(r_{jx}\Bar{r_j}){|\phib+r_j|}^{p-3}(\phib+r_j) +r_{jx}\left[ {|\phib+r_j|}^{p-1} -{|R_j|}^{p-1}\right]\\ &\quad +\phib_x\left[ {|\phib+r_j|}^{p-1}-{|\phib|}^{p-1} -(p-1)\re(\Bar\phib r_j){|\phib|}^{p-3}\right]\\ &\quad +(p-1)\left[ \re(\Bar\phib r_j)\phib_x {|\phib|}^{p-3} -\re(\Bar{R_j}r_j)R_{jx} {|R_j|}^{p-3}\right]. \end{align*} To estimate all these terms in $L^2$ norm, we use the facts that $\phib$ is equal to $R$ plus a small error term according to \eqref{eq:cmmNLS}, that $R$ multiplied by a term moving on the line $x=v_jt+x_j$ (like $r_j$) is equal to $R_j$ plus a small error term according to \eqref{eq:interactNLS}, and finally that $r_j$ is at order $e^{-e_jt}$. To illustrate this, we estimate the first two terms $\mathbf{I}$ and $\mathbf{II}$, for example, as all other terms can be treated similarly. For $\mathbf{I}$, we simply remark that \[ \nld{\mathbf{I}} \leq C\nld{r_j}^2 \leq Ce^{-2e_jt}\leq Ce^{-(e_j+4\g)t} \] by the definition of $\g$ \eqref{gammaNLS}. For $\mathbf{II}$, we decompose it as \begin{align*} \frac{1}{(p-1)(p-3)}\mathbf{II} &= \re[(\phib_x-R_x)\Bar\phib] \re(\Bar\phib r_j){|\phib|}^{p-5}\phib + \re(R_x(\Bar\phib-\Bar R) \re(\Bar\phib r_j){|\phib|}^{p-5}\phib\\ &\quad + \re(R_x\Bar R) \re[(\Bar\phib-\Bar R) r_j]{|\phib|}^{p-5}\phib + \re[(R_x-R_{jx})\Bar R] \re(\Bar R r_j){|\phib|}^{p-5}\phib\\ &\quad + \re[R_{jx}(\Bar R-\Bar{R_j})] \re(\Bar R r_j){|\phib|}^{p-5}\phib + \re[R_{jx}\Bar{R_j}] \re[(\Bar R -\Bar{R_j})r_j]{|\phib|}^{p-5}\phib\\ &\quad +\re(R_{jx}\Bar{R_j})\re(\Bar{R_j}r_j) \left[ {|\phib|}^{p-5}\phib -{|R_j|}^{p-5}R_j \right]. \end{align*} Since $\nh{\phib-R}\leq Ce^{-4\g t}$ by \eqref{eq:cmmNLS}, the first three terms are bounded in $L^2$ norm by $Ce^{-(e_j+4\g)t}$. Moreover, by \eqref{eq:interactNLS}, the next three terms are also bounded in $L^2$ norm by $Ce^{-(e_j+4\g)t}$. Finally, for the last term, we write \[ {|\phib|}^{p-5}\phib -{|R_j|}^{p-5}R_j = ({|\phib|}^{p-5}\phib - {|R|}^{p-5}R) + ({|R|}^{p-5}R -{|R_j|}^{p-5}R_j), \] so that, since $p>5$, we can conclude similarly that $\nld{\mathbf{II}}\leq Ce^{-(e_j+4\g)t}$.
\end{proof}

\begin{proof}[Proof of Lemma \ref{th:h}]
\begin{enumerate}[(i)]
\item For $k\in\unn$, we have \begin{align*} (|R_k|+|R_{kx}|)|\phi_k-1| &\leq Ce^{-\sqrt{c_k}|x-v_kt|}[1+\psi_{k+1}-\psi_k]\\ &\leq Ce^{-\sqrt{\sigma_0}|x-v_kt|}\cdot e^{-\sqrt{\sigma_0}|x-v_kt|}[1+\psi_{k+1}-\psi_k]. \end{align*} But, if $x<m_k(t)+\sqrt t$, then \[ e^{-\sqrt{\sigma_0}|x-v_kt|}[1+\psi_{k+1}-\psi_k] \leq Ce^{\sqrt{\sigma_0}x}e^{-\sqrt{\sigma_0}v_kt}\leq Ce^{\frac{1}{2}\sqrt{\sigma_0}(v_k+v_{k-1}-2v_k)t}e^{\sqrt{\sigma_0}\sqrt{t}} \leq Ce^{-\frac{1}{4}\sigma_0^{3/2}t}, \] and similarly, if $x>m_{k+1}(t)-\sqrt t$, then \[ e^{-\sqrt{\sigma_0}|x-v_kt|}[1+\psi_{k+1}-\psi_k] \leq Ce^{-\sqrt{\sigma_0}x}e^{\sqrt{\sigma_0}v_kt} \leq Ce^{-\frac{1}{2}\sqrt{\sigma_0}(v_{k+1}-v_k-2v_k)t}e^{\sqrt{\sigma_0}\sqrt{t}} \leq Ce^{-\frac{1}{4}\sigma_0^{3/2}t}. \] As $\phi_k(t,x)=1$ for $m_k(t)+\sqrt t\leq x\leq m_{k+1}(t)-\sqrt t$, the conclusion follows from \eqref{gammaNLS}.
\item For $l,k\in\unn$ such that $l\neq k$, we have \begin{align*} (|R_k|+|R_{kx}|)\phi_l &\leq Ce^{-\sqrt{c_k}|x-v_kt|}[\psi_l-\psi_{l+1}]\mathbbm{1}_{\{x>m_l(t)-\sqrt t\}} \mathbbm{1}_{\{x<m_{l+1}(t)+\sqrt t\}}\\ &\leq Ce^{-\sqrt{\sigma_0}|x-v_kt|}\cdot e^{-\sqrt{\sigma_0}|x-v_kt|} \mathbbm{1}_{\{x>m_l(t)-\sqrt t\}} \mathbbm{1}_{\{x<m_{l+1}(t)+\sqrt t\}}. \end{align*} But, if $k>l$, then \begin{align*} e^{-\sqrt{\sigma_0}|x-v_kt|} \mathbbm{1}_{\{x>m_l(t)-\sqrt t\}} \mathbbm{1}_{\{x<m_{l+1}(t)+\sqrt t\}} &\leq e^{\sqrt{\sigma_0}x}e^{-\sqrt{\sigma_0}v_kt} \mathbbm{1}_{\{x>m_l(t)-\sqrt t\}} \mathbbm{1}_{\{x<m_{l+1}(t)+\sqrt t\}}\\ &\leq Ce^{\frac{1}{2}\sqrt{\sigma_0}(v_{l+1}+v_l-2v_k)t}e^{\sqrt{\sigma_0}\sqrt t} \leq Ce^{-\frac{1}{4}\sigma_0^{3/2}t}, \end{align*} and similarly, if $k<l$, then \begin{align*} e^{-\sqrt{\sigma_0}|x-v_kt|} \mathbbm{1}_{\{x>m_l(t)-\sqrt t\}} \mathbbm{1}_{\{x<m_{l+1}(t)+\sqrt t\}} &\leq Ce^{-\sqrt{\sigma_0}x}e^{\sqrt{\sigma_0}v_kt} \mathbbm{1}_{\{x>m_l(t)-\sqrt t\}} \mathbbm{1}_{\{x<m_{l+1}(t)+\sqrt t\}}\\ &\leq Ce^{-\frac{1}{2}\sqrt{\sigma_0}(v_l+v_{l-1}-2v_k)t}e^{\sqrt{\sigma_0}\sqrt{t}} \leq Ce^{-\frac{1}{4}\sigma_0^{3/2}t}, \end{align*} and the conclusion follows again from the definition of $\g$.
\item For $k\in\unn$, it suffices to prove $\nli{\psi_{kx}}+\nli{\psi_{kxx}}+\nli{\psi_{kt}}\leq \frac{C}{\sqrt t}$. The first two inequalities are obvious since $\psi_{kx}(t,x)= \frac{1}{\sqrt t}\psi'\left[\frac{1}{\sqrt t}(x-m_k(t))\right]$ and so $\nli{\psi_{kx}}\leq \frac{1}{\sqrt t}\nli{\psi'}$, and similarly $\nli{\psi_{kxx}}\leq \frac{1}{t}\nli{\psi''}$. For the last one, we write \[ \psi_k(t,x) = \psi\left[ \frac{x-\frac{1}{2}(x_k+x_{k-1})}{\sqrt t} -\frac{1}{2}(v_k+v_{k-1})\sqrt t\right], \] so that \[ \psi_{kt}(t,x) = \left[ -\frac{1}{2}\left( \frac{x-\frac{x_k+x_{k-1}}{2}}{t^{3/2}} \right) -\frac{1}{4}\left( \frac{v_k+v_{k-1}}{\sqrt t}\right) \right]\cdot \psi'\left[ \frac{1}{\sqrt t}(x-m_k(t))\right]\mathbbm{1}_{|x-m_k(t)|\leq \sqrt t}, \] since $\mathrm{supp}(\psi')=[-1,1]$. But for $x$ such that $|x-m_k(t)|\leq \sqrt t$, we have $\left| x-\frac{x_k+x_{k-1}}{2}\right|\leq Ct$, and so finally $\nli{\psi_{kt}}\leq \frac{C}{\sqrt t}\nli{\psi'}$.
\item Since $h_1\equiv \sum_{k=1}^N \left( c_k+\frac{v_k^2}{4}\right)\phi_k$ and $h_2\equiv \sum_{k=1}^N v_k\phi_k$ have a similar form, it is clear that it suffices to prove the inequalities for $h_2$, for example. Moreover, the first inequalities are obvious by~(iii). Finally, for the last inequality, we write \begin{multline*} |h_2-v_k|(|R_k|+|R_{kx}|) = \left| \sum_{l=1}^N v_l\phi_l -v_k\right|(|R_k|+|R_{kx}|)\\ \leq v_k|\phi_k-1|(|R_k|+|R_{kx}|) +\sum_{l\neq k} v_l\phi_l(|R_k|+|R_{kx}|) \leq Ce^{-4\g t}e^{-\sqrt{\sigma_0}|x-v_kt|} \end{multline*} by (i) and (ii), which concludes the proof. \qedhere
\end{enumerate}
\end{proof}

\begin{proof}[Proof of Lemma \ref{th:lienHH}]
To compare $\H[\zt]$ and $\H[z]$, we replace $\zt$ in $\H[\zt]$ by its definition, \[ \zt = z+\sum_{k=1}^N \beta_k iQ_{c_k}(\lambda_k)e^{i\theta_k} +\sum_{k=1}^N \g_k\px Q_{c_k}(\lambda_k)e^{i\theta_k}, \] dropping the argument $\lambda_k$ for this proof, which would not be a source of confusion since there is no time derivative. Hence, we compute \begin{align*} \H[\zt] &= \int \px\zt\cdot \Bar{\px\zt} -\im h_2\px\zt\cdot\Bar{\zt} +(h_1-{|R|}^{p-1})\zt\cdot\Bar{\zt} -(p-1)\Carre{\re(\Bar R\zt)}{|R|}^{p-3}\\ &= \int \left[ \px z +\sum \left( \g_k\px^2Q_{c_k}-\frac{\beta_k}{2}v_kQ_{c_k}+i\px Q_{c_k}(\beta_k+\frac{1}{2}v_k\g_k)\right) e^{i\theta_k}\right]\\ &\qquad \times \left[ \px\bar z +\sum\left( \g_k\px^2 Q_{c_k} -\frac{\beta_k}{2}v_kQ_{c_k}-i\px Q_{c_k}(\beta_k+\frac{1}{2}v_k\g_k)\right)e^{-i\theta_k}\right]\\ &\quad -\int h_2\im\left[ \px z+\sum \left( \g_k\px^2 Q_{c_k}-\frac{\beta_k}{2}v_kQ_{c_k} +i\px Q_{c_k}(\beta_k+\frac{1}{2}v_k\g_k)\right) e^{i\theta_k}\right]\\ &\qquad \times \left[ \bar z+\sum (\g_k\px Q_{c_k}-i\beta_kQ_{c_k})e^{-i\theta_k}\right]\\ &\quad +\int (h_1-{|R|})^{p-1}\left[ z+\sum (\g_k\px Q_{c_k} +i\beta_k Q_{c_k})e^{i\theta_k}\right]\times \left[ \bar z+\sum (\g_k\px Q_{c_k}-i\beta_kQ_{c_k})e^{-i\theta_k}\right]\\ &\quad -\int (p-1){|R|}^{p-3} {\left[ \re(\Bar Rz)-\sum \beta_k\im(R_k\Bar R) +\sum \g_k\re(\px Q_{c_k}e^{i\theta_k}\Bar R)\right]}^2. \end{align*} Developing in terms of $z$, we find \begin{align*} \H[\zt] &= \int {|\px z|}^2 +2\re\int \px z\cdot \sum\left( \g_k\px^2Q_{c_k}-\frac{\beta_k}{2}v_kQ_{c_k}-i\px Q_{c_k}(\beta_k+\frac{1}{2}v_k\g_k)\right) e^{-i\theta_k}\\ &\quad +\sum_{k,l}\int \left( \g_k\px^2Q_{c_k}-\frac{\beta_k}{2}v_kQ_{c_k}+i\px Q_{c_k}(\beta_k+\frac{1}{2}v_k\g_k)\right) e^{i\theta_k}\\ &\qquad\qquad \times \left(\g_l\px^2Q_{c_l}-\frac{\beta_l}{2}v_lQ_{c_l}-i\px Q_{c_l}(\beta_l+\frac{1}{2}v_l\g_l)\right) e^{-i\theta_l}\\ &\quad -\im\int h_2\px z\cdot\bar z -\im\int h_2\px z\cdot\sum (\g_k\px Q_{c_k}-i\beta_kQ_{c_k})e^{-i\theta_k}\\ &\quad +\im\int h_2z\cdot\sum \left( \g_k\px^2Q_{c_k}-\frac{\beta_k}{2}v_kQ_{c_k}-i\px Q_{c_k}(\beta_k+\frac{1}{2}v_k\g_k)\right) e^{-i\theta_k}\\ &\quad -\sum_{k,l} \im\int h_2 \left( \g_k\px^2Q_{c_k}-\frac{\beta_k}{2}v_kQ_{c_k}+i\px Q_{c_k}(\beta_k+\frac{1}{2}v_k\g_k)\right) e^{i\theta_k} (\g_l\px Q_{c_l}-i\beta_lQ_{c_l})e^{-i\theta_l}\\ &\quad +\int (h_1-{|R|}^{p-1}){|z|}^2 +2\re\int (h_1-{|R|}^{p-1})z\cdot\sum (\g_k\px Q_{c_k}-i\beta_kQ_{c_k})e^{-i\theta_k}\\ &\quad +\sum_{k,l}\int (h_1-{|R|}^{p-1}) (\g_k\px Q_{c_k}+i\beta_kQ_{c_k})e^{i\theta_k}(\g_l\px Q_{c_l}-i\beta_lQ_{c_l})e^{-i\theta_l}\\ &\quad -(p-1)\int {|R|}^{p-3}\Carre{\re(\Bar Rz)} -(p-1)\int {|R|}^{p-3}\sum_{k,l} \beta_k\beta_l \im(R_k\Bar R)\im(R_l\Bar R)\\ &\quad -(p-1)\int {|R|}^{p-3} \sum_{k,l} \g_k\g_l \re(\px Q_{c_k}e^{i\theta_k}\Bar R)\re(\px Q_{c_l}e^{i\theta_l}\Bar R)\\ &\quad +2(p-1)\int {|R|}^{p-3}\re(\Bar Rz)\sum \beta_k\im(R_k\Bar R)\\ &\quad -2(p-1)\int {|R|}^{p-3} \re(\Bar Rz)\sum \g_k\re(\px Q_{c_k}e^{i\theta_k}\Bar R)\\ &\quad +2(p-1)\int {|R|}^{p-3} \sum_{k,l} \beta_k\g_l \im(R_k\Bar R)\re(\px Q_{c_l}e^{i\theta_l}\Bar R). \end{align*} Now, first remark that $\im (R_k\Bar R) = \sum_{q\neq k} \im(R_k\Bar{R_q})$, and so, by \eqref{eq:interactNLS}, all integrals containing this term are in $O(e^{-\g t}\nh{z}^2)$. Moreover, still by \eqref{eq:interactNLS}, all double sums on $k,l$ have their terms in $O(e^{-\g t}\nh{z}^2)$ whenever $k\neq l$. Note finally that all terms composing $\H[z]$ appear. Hence, with an integration by parts to make $\px z$ disappear, we have \begin{align*} H[\zt] &= \int {|\px z|}^2 -\im h_2\px z\cdot\bar z +(h_1-{|R|}^{p-1}){|z|}^2 -(p-1){|R|}^{p-3}\Carre{\re(\Bar Rz)} +O(e^{-\g t}\nh{z}^2)\\ &\quad -2\sum\re\int z e^{-i\theta_k}\left[ \left(\g_k\px^3Q_{c_k} -\beta_kv_k\px Q_{c_k} -\frac{1}{4}\g_kv_k^2\px Q_{c_k}\right) \right.\\ &\qquad \left. +i\left( -v_k\g_k\px^2Q_{c_k} -\beta_k\px^2 Q_{c_k} +\frac{1}{4}v_k^2\beta_kQ_{c_k}\right) \right]\\ &\quad +\sum\int \Carre{\g_k\px^2 Q_{c_k}-\frac{\beta_k}{2}v_kQ_{c_k}} +\Carre{\beta_k+\frac{1}{2}v_k\g_k}{(\px Q_{c_k})}^2 \\ &\quad +\sum\im\int z\px h_2(\g_k\px Q_{c_k}-i\beta_kQ_{c_k})e^{-i\theta_k}\\ &\quad +2\sum\im\int h_2ze^{-i\theta_k}\left[ \left( \g_k\px^2 Q_{c_k} -\frac{\beta_k}{2}v_kQ_{c_k}\right) -i\px Q_{c_k}\left( \beta_k+\frac{1}{2}v_k\g_k\right) \right]\\ &\quad -\sum\int h_2\g_k(\beta_k+\frac{1}{2}v_k\g_k){(\px Q_{c_k})}^2 +\sum\int h_2 \beta_k Q_{c_k}(\g_k\px^2 Q_{c_k}-\frac{\beta_k}{2}v_k Q_{c_k})\\ &\quad +2\sum\re \int (h_1-{|R|}^{p-1})ze^{-i\theta_k}(\g_k\px Q_{c_k}-i\beta_k Q_{c_k})\\ &\quad +\sum\int (h_1-{|R|}^{p-1})(\g_k^2{(\px Q_{c_k})}^2 +\beta_k^2Q_{c_k}^2) \\ &\quad -(p-1)\sum \int {|R|}^{p-3} \g_k^2 Q_{c_k}^2 {(\px Q_{c_k})}^2 -2(p-1)\sum\re\int {|R|}^{p-3}ze^{-i\theta_k}\g_k Q_{c_k}^2 \px Q_{c_k}. \end{align*}

We now use notation $z_{1,k}=\re(z^{-i\theta_k})$ and $z_{2,k} = \im(z^{-i\theta_k})$ again. Moreover, recall that we have $\nli{\px h_2}\leq \frac{C}{\sqrt t}$ by (iv) of Lemma \ref{th:h}, and $\px^2 Q_{c_k}+Q_{c_k}^p = c_kQ_{c_k}$ by \eqref{eq:Qc}. Thus, we find \begin{align} \H[\zt] &= \H[z] +O(t^{-1/2}\nh{z}^2) \notag \\ &\quad +\sum\int z_{1,k} [-2c_k\g_k\px Q_{c_k} +2p\g_k\px Q_{c_k}Q_{c_k}^{p-1} +2\beta_kv_k\px Q_{c_k}+\frac{1}{2}\g_kv_k^2\px Q_{c_k} \notag \\ &\qquad -2h_2\beta_k\px Q_{c_k} -h_2\g_kv_k\px Q_{c_k} +2h_1\g_k\px Q_{c_k}-2\g_k\px Q_{c_k}Q_{c_k}^{p-1} -2(p-1)\g_k\px Q_{c_k}Q_{c_k}^{p-1}] \label{eq:z1k} \\ &\quad +\sum\int z_{2,k} [ -2\g_kv_kc_kQ_{c_k} +2\g_kv_kQ_{c_k}^p -2\beta_kc_kQ_{c_k} +2\beta_k Q_{c_k}^p +\frac{1}{2}\beta_kv_k^2Q_{c_k} +2h_2\g_kc_kQ_{c_k} \notag \\* &\qquad -2h_2\g_k Q_{c_k}^p -h_2\beta_kv_kQ_{c_k} +2h_1\beta_kQ_{c_k} -2\beta_k Q_{c_k}^p] \label{eq:z2k} \\ &\quad +\sum\int \Carre{\g_kc_kQ_{c_k} -\g_kQ_{c_k}^p -\frac{\beta_k}{2}v_kQ_{c_k}} +\Carre{\beta_k+\frac{1}{2}v_k\g_k}{(\px Q_{c_k})}^2 \notag \\ &\quad -\sum\int h_2\g_k(\beta_k+\frac{1}{2}v_k\g_k){(\px Q_{c_k})}^2 +\sum\int h_2 \beta_k Q_{c_k}(\g_kc_k Q_{c_k}-\g_kQ_{c_k}^p -\frac{\beta_k}{2}v_k Q_{c_k}) \notag \\ &\quad + \sum\int h_1[\g_k^2 {(\px Q_{c_k})}^2 +\beta_k^2Q_{c_k}^2] -\sum\int Q_{c_k}^{p-1}[\g_k^2 {(\px Q_{c_k})}^2 +\beta_k^2Q_{c_k}^2] \notag \\ &\quad -\sum\int (p-1)\g_k^2 Q_{c_k}^{p-1}{(\px Q_{c_k})}^2. \label{eq:sources} \end{align}

To conclude, we estimate the term \eqref{eq:z1k} involving $z_{1,k}$, the term \eqref{eq:z2k} involving $z_{2,k}$, and finally the source term \eqref{eq:sources}. For \eqref{eq:z1k}, we write \[ \eqref{eq:z1k} = \sum\int z_{1,k}\g_k\px Q_{c_k}(-2c_k+\frac{v_k^2}{2}-h_2v_k+2h_1) +2\sum\int z_{1,k}\beta_k\px Q_{c_k}(v_k-h_2), \] and $-2c_k+\frac{v_k^2}{2}-h_2v_k+2h_1 = 2(h_1-c_k-\frac{v_k^2}{4}) +v_k(v_k-h_2)$, so that, by (iv) of Lemma \ref{th:h}, we have $\eqref{eq:z1k}=O(e^{-\g t}\nh{z}^2)$. Similarly, we write \begin{multline*} \eqref{eq:z2k} = 2\sum\int z_{2,k}\g_kc_k Q_{c_k}(h_2-v_k) +2\sum\int z_{2,k}\g_k Q_{c_k}^p (v_k-h_2)\\ +\sum\int z_{2,k}\beta_k Q_{c_k}(-2c_k+\frac{v_k^2}{2}-h_2v_k+2h_1), \end{multline*} and we also conclude that $\eqref{eq:z2k}=O(e^{-\g t}\nh{z}^2)$. For the last term, we expand it as \begin{align*} &\eqref{eq:sources} = \sum\int \beta_k\g_kc_k Q_{c_k}^2(h_2-v_k) +\beta_k\g_kQ_{c_k}^{p+1}(v_k-h_2) +\beta_k\g_k{(\px Q_{c_k})}^2(v_k-h_2)\\ & + \sum\int \g_k^2c_k^2Q_{c_k}^2 +\g_k^2Q_{c_k}^{2p} +\frac{\beta_k^2}{4}v_k^2Q_{c_k}^2 -2\g_k^2c_kQ_{c_k}^{p+1} +\beta_k^2 {(\px Q_{c_k})}^2 +\frac{1}{4}\g_k^2v_k^2{(\px Q_{c_k})}^2\\ &\ -\frac{1}{2}h_2\g_k^2v_k{(\px Q_{c_k})}^2 -\frac{1}{2}h_2\beta_k^2v_kQ_{c_k}^2 +h_1\g_k^2{(\px Q_{c_k})}^2 +h_1\beta_k^2Q_{c_k}^2 -\beta_k^2Q_{c_k}^{p+1} -p\g_k^2 Q_{c_k}^{p-1}{(\px Q_{c_k})}^2. \end{align*} Note that the first sum is in $O(e^{-\g t}\nh{z}^2)$ as above. Hence, with several integrations by parts and using $\px^2 Q_{c_k}=c_kQ_{c_k}-Q_{c_k}^p$, we find \begin{align*} \eqref{eq:sources} &= O(e^{-\g t}\nh{z}^2) +\sum\int \g_k^2c_k^2Q_{c_k}^2 +\g_k^2Q_{c_k}^{2p} +\frac{\beta_k^2}{4}v_k^2Q_{c_k}^2 -2\g_k^2c_kQ_{c_k}^{p+1} -\beta_k^2Q_{c_k}(c_kQ_{c_k}-Q_{c_k}^p)\\ &\quad -\frac{1}{4}\g_k^2v_k^2Q_{c_k}(c_kQ_{c_k}-Q_{c_k}^p) +\frac{1}{2}h_2\g_k^2v_kQ_{c_k}(c_kQ_{c_k}-Q_{c_k}^p) -\frac{1}{2}h_2\beta_k^2v_kQ_{c_k}^2\\ &\quad -h_1\g_k^2Q_{c_k}(c_kQ_{c_k}-Q_{c_k}^p)+h_1\beta_k^2Q_{c_k}^2 -\beta_k^2Q_{c_k}^{p+1} +\g_k^2Q_{c_k}^p(c_kQ_{c_k}-Q_{c_k}^p)\\ &= O(e^{-\g t}\nh{z}^2)- \frac{1}{2}\sum\int \g_k^2c_kQ_{c_k}^2(-2c_k+\frac{v_k^2}{2}-h_2v_k+2h_1)\\ &\quad +\frac{1}{2}\sum\int \beta_k^2Q_{c_k}^2(-2c_k+\frac{v_k^2}{2}-h_2v_k+2h_1) +\frac{1}{2}\sum\int \g_k^2 Q_{c_k}^{p+1}(-2c_k+\frac{v_k^2}{2}-h_2v_k+2h_1), \end{align*} and so we can conclude as above that $\eqref{eq:sources} = O(e^{-\g t}\nh{z}^2)$. Finally, we proved that $\H[\zt]= \H[z]+O(t^{-1/2}\nh{z}^2)$, as expected.
\end{proof}

\section{Appendix} \label{app:dHdt}

We prove here Proposition \ref{th:dHdtNLS}. To do this, we first need a lemma quantifying the fact that $\phib$ almost satisfies a transport equation similar to those satisfied by the solitons. Note finally that, since $\phib_t$ takes values in $H^{-1}$, all integrals in this appendix may be seen as the dual bracket ${\langle \cdot,\cdot\rangle}_{H^1,H^{-1}}$.

\begin{lem} \label{th:transport}
There exists $C>0$ such that, for all $t\geq T_0$, ${\|\phib_t+h_2\phib_x-ih_1\phib\|}_{H^{-1}} \leq Ce^{-4\g t}$.
\end{lem}

\begin{rem}
To find the transport equation almost satisfied by $\phib$, it suffices to compute an exact relation for $R_k$ with $k\in\unn$. In fact, as \[ R_k(t,x)= Q_{c_k}(x-v_kt-x_k)e^{i(\frac{1}{2}v_kx -\frac{1}{4}v_k^2t+c_kt+\g_k)}, \] we have $R_{kt} = [-v_k\px Q_{c_k}+i(c_k-\frac{1}{4}v_k^2)Q_{c_k}](\lambda_k)e^{i\theta_k}$ and $R_{kx} = [\px Q_{c_k}+\frac{i}{2}v_kQ_{c_k}](\lambda_k)e^{i\theta_k}$, and so \[ R_{kt} +v_kR_{kx} -i\left( c_k+\frac{v_k^2}{4}\right)R_k = 0. \]
\end{rem}

\begin{proof}[Proof of Lemma \ref{th:transport}]
Let $f\in H^1$ and compute \begin{align*} \int &(\phib_t+h_2\phib_x-ih_1\phib)f = \int (i\phib_{xx}+i{|\phib|}^{p-1}\phib +h_2\phib_x-ih_1\phib)f\\ &= i\int (\phib_{xx}-R_{xx})f +i\int ({|\phib|}^{p-1}\phib -{|R|}^{p-1}R)f +\int h_2(\phib_x-R_x)f -i\int h_1(\phib-R)f\\ &\quad +i\int (R_{xx}+{|R|}^{p-1}R-ih_2R_x -h_1R)f\\ &= -i\int (\phib_x-R_x)f_x +i\int ({|\phib|}^{p-1}\phib -{|R|}^{p-1}R)f +\int h_2(\phib_x-R_x)f -i\int h_1(\phib-R)f\\ &\quad + i\sum_{k=1}^N \int (R_{kxx}+{|R_k|}^{p-1}R_k-ih_2R_{kx}-h_1R_k)f +i\sum_{k=1}^N\int R_k({|R|}^{p-1}-{|R_k|}^{p-1})f\\ &= \mathbf{I} + \mathbf{II} +\mathbf{III}. \end{align*} First note that, by \eqref{eq:cmmNLS}, $|\mathbf{I}|\leq C\nh{\phib-R}\nh{f}\leq Ce^{-4\g t}\nh{f}$. Moreover, by \eqref{eq:interactNLS}, we also have $|\mathbf{III}|\leq Ce^{-4\g t}\nld{f}$. For the last term, we first compute \[ \begin{cases} R_k = Q_{c_k}(\lambda_k)e^{i\theta_k},\ R_{kx} = (\px Q_{c_k}+\frac{i}{2}v_kQ_{c_k})(\lambda_k)e^{i\theta_k},\\ R_{kxx} = (\px^2 Q_{c_k}+iv_k\px Q_{c_k} -\frac{v_k^2}{4}Q_{c_k})(\lambda_k)e^{i\theta_k}, \end{cases} \] and so, using $\px^2 Q_{c_k} = c_kQ_{c_k}-Q_{c_k}^p$, we obtain \begin{align*} \mathbf{II} &= i\sum_{k=1}^N\int \left[\left( c_k-\frac{v_k^2}{4}-h_1\right)R_k +iv_kR_{kx} +\frac{v_k^2}{2}R_k-ih_2R_{kx}\right]f\\ &= i\sum_{k=1}^N \int \left(c_k+\frac{v_k^2}{4} -h_1\right)R_kf +\sum_{k=1}^N \int (h_2-v_k)R_{kx}f. \end{align*} Therefore, by (iv) of Lemma \ref{th:h}, we also have $|\mathbf{II}|\leq Ce^{-4\g t}\nld{f}$, which concludes the proof of Lemma \ref{th:transport}.
\end{proof}

\begin{proof}[Proof of Proposition \ref{th:dHdtNLS}]
First recall that, from Section \ref{sec:equationofz}, the equation of $z$ can be written \[ iz_t+z_{xx} +{|\phib+r_j+z|}^{p-1}(\phib+r_j+z) -{|\phib+r_j|}^{p-1}(\phib+r_j)=-\Omega, \] where $r_j(t,x)=A_je^{-e_jt}Y_j^+(t,x)$ and $\Omega$ satisfies $\nh{\Omega}\leq Ce^{-(e_j+4\g)t}$ by Lemma \ref{th:Omegadecroit}.

From the definition of $H$ \eqref{eq:defH}, we now compute, using integrations by parts, \begin{align*} H'(t) &= 2\re\int z_{tx}\bar{z}_x -2\re\int {(\phib+r_j+z)}_t{|\phib+r_j+z|}^{p-1}(\Bar\phib +\Bar{r_j}+\bar z)\\ &\quad +2\re\int {(\phib+r_j)}_t{|\phib+r_j|}^{p-1}(\Bar\phib+\Bar{r_j})\\ &\quad +2(p-1)\re\int {(\phib+r_j)}_t {|\phib+r_j|}^{p-3}(\Bar\phib+\Bar{r_j})\re[(\Bar\phib+\Bar{r_j})z]\\ &\quad +2\int {|\phib+r_j|}^{p-1}\re[{(\Bar\phib+\Bar{r_j})}_tz] +2\int {|\phib+r_j|}^{p-1}\re[(\Bar\phib+\Bar{r_j})z_t]\\ &\quad +\int h_{1t}{|z|}^2 +2\re\int h_1z_t\bar z -\im\int h_{2t}z_x\bar z -\im\int h_2z_{tx}\bar z-\im\int h_2z_x\bar{z}_t\\ &= -2\re\int z_t\left[ \bar{z}_{xx} + {|\phib+r_j+z|}^{p-1}(\Bar\phib +\Bar{r_j}+\bar z) -{|\phib+r_j|}^{p-1}(\Bar\phib+\Bar{r_j})\right]\\ &\quad -2\re\int {(\phib+r_j)}_t \left[ {|\phib+r_j+z|}^{p-1}(\Bar\phib +\Bar{r_j}+\bar z) \right. \\ &\qquad\qquad \left. - {|\phib+r_j|}^{p-1}(\Bar\phib+\Bar{r_j}+\bar z) -(p-1){|\phib+r_j|}^{p-3}(\Bar\phib+\Bar{r_j}) \re[(\Bar\phib+\Bar{r_j})z] \right]\\ &\quad +2\re\int h_1z_t\bar z +2\im\int h_2\bar{z}_xz_t +\im\int h_{2x}z_t\bar z +\int h_{1t}{|z|}^2 -\im\int h_{2t}z_x\bar z. \end{align*} But from (iv) of Lemma \ref{th:h}, we have $\nli{h_{1t}} +\nli{h_{2t}}\leq \frac{C}{\sqrt t}$, and so \[ \left| \int h_{1t}{|z|}^2 -\im\int h_{2t}z_x\bar z \right| \leq \frac{C}{\sqrt t}\nh{z}^2. \] Moreover, by expanding ${|\phib+r_j+z|}^{p-1} = {\left[|\phib+r_j+z|^2\right]}^{\frac{p-1}{2}}$, and as $\nli{r_{jt}}\leq Ce^{-e_jt}$, we have \begin{multline*} \left| -2\re\int r_{jt} \left[ {|\phib+r_j+z|}^{p-1}(\Bar\phib +\Bar{r_j}+\bar z) - {|\phib+r_j|}^{p-1}(\Bar\phib+\Bar{r_j}+\bar z) \right. \right.\\ \left. \left. -(p-1){|\phib+r_j|}^{p-3}(\Bar\phib+\Bar{r_j}) \re[(\Bar\phib+\Bar{r_j})z] \right] \right| \leq Ce^{-4\g t}\nh{z}^2. \end{multline*}

Hence, replacing $z_t$ by its equation, we find \begin{align*} H'(t) &= -2\im\int \Bar{\Omega} \left[z_{xx} +{|\phib+r_j+z|}^{p-1}(\phib+r_j+z) -{|\phib+r_j|}^{p-1}(\phib+r_j)\right]\\ &\quad -2\re\int \phib_t \left[ {|\phib+r_j+z|}^{p-1}(\Bar\phib +\Bar{r_j}+\bar z)\right. \\ &\qquad\qquad \left. - {|\phib+r_j|}^{p-1}(\Bar\phib+\Bar{r_j}+\bar z) -(p-1){|\phib+r_j|}^{p-3}(\Bar\phib+\Bar{r_j}) \re[(\Bar\phib+\Bar{r_j})z] \right]\\ &\quad -2\im\int h_1\bar zz_{xx} -2\im\int h_1\bar z[{|\phib+r_j+z|}^{p-1}(\phib+r_j+z) -{|\phib+r_j|}^{p-1}(\phib+r_j)]\\ &\quad -2\im\int h_1\Omega\bar z +2\re\int h_2\bar{z}_xz_{xx} +\re\int h_{2x}\bar zz_{xx} +\re\int (2h_2\bar{z}_x+h_{2x}\bar z)\Omega\\ &\quad -2\re\int h_2\bar{z}{\left[{|\phib+r_j+z|}^{p-1}(\phib+r_j+z) -{|\phib+r_j|}^{p-1}(\phib+r_j)\right]}_x\\ &\quad -\re\int h_{2x}\bar z\left[{|\phib+r_j+z|}^{p-1}(\phib+r_j+z) -{|\phib+r_j|}^{p-1}(\phib+r_j)\right] +O(t^{-1/2}\nh{z}^2). \end{align*} We can already estimate several terms in this expression. For the first term, for example, we have, by an integration by parts, \[ \left| -2\im\int \Bar{\Omega}z_{xx}\right| = \left| 2\im\int \Bar{\Omega}_xz_x\right| \leq C\nh{\Omega}\nh{z}\leq Ce^{-(e_j+4\g)t}\nh{z}. \] Similarly, we have \begin{gather*} \left| -2\im\int \Bar{\Omega} \left[ {|\phib+r_j+z|}^{p-1}(\phib+r_j+z) -{|\phib+r_j|}^{p-1}(\phib+r_j)\right] \right| \leq Ce^{-(e_j+4\g)t}\nh{z},\\ \left| -2\im\int h_1\Omega\bar z +\re\int (2h_2\bar{z}_x+h_{2x}\bar z)\Omega\right| \leq Ce^{-(e_j+4\g)t}\nh{z}. \end{gather*} Then, another integration by parts gives \[ -2\im\int h_1\bar zz_{xx} = 2\im\int h_1{|z_x|}^2 +2\im\int h_{1x}\bar zz_x = 2\im\int h_{1x}\bar zz_x, \] and so, as $\nli{h_{1x}}\leq \frac{C}{\sqrt t}$ by Lemma \ref{th:h}, $\left| -2\im\int h_1\bar zz_{xx}\right|\leq \frac{C}{\sqrt t}\nh{z}^2$. As we also have $\nli{h_{2x}}\leq \frac{C}{\sqrt t}$, we can estimate \[ \left| -\re\int h_{2x}\bar z\left[{|\phib+r_j+z|}^{p-1}(\phib+r_j+z) -{|\phib+r_j|}^{p-1}(\phib+r_j)\right] \right|\leq \frac{C}{\sqrt t}\nh{z}^2. \] Finally, we can also estimate \begin{align*} 2\re\int h_2\bar{z}_xz_{xx} +\re\int h_{2x}\bar zz_{xx} &= -\int h_{2x}{|z_x|}^2 -\re\int z_x(h_{2xx}\bar z+h_{2x}\bar{z}_x)\\ &= -2\int h_{2x}{|z_x|}^2 -\re\int h_{2xx}z_x\bar z. \end{align*} Indeed, since $\nli{h_{2x}}+\nli{h_{2xx}}\leq \frac{C}{\sqrt t}$ by Lemma \ref{th:h}, we have \[ \left| 2\re\int h_2\bar{z}_xz_{xx} +\re\int h_{2x}\bar zz_{xx} \right| \leq \frac{C}{\sqrt t}\nh{z}^2. \]

Gathering all previous estimates, we have proved that \[ -\frac{1}{2} H'(t) = \mathbf{I} +\mathbf{II} +\mathbf{III} + O(e^{-(e_j+4\g)t}\nh{z}) +O(t^{-1/2}\nh{z}^2), \] where \[\left\{ \begin{aligned} \mathbf{I} &= \re\int h_2\bar{z}{\left[{|\phib+r_j+z|}^{p-1}(\phib+r_j+z) -{|\phib+r_j|}^{p-1}(\phib+r_j)\right]}_x,\\ \mathbf{II} &= \im\int h_1\bar z\left[{|\phib+r_j+z|}^{p-1}(\phib+r_j+z) -{|\phib+r_j|}^{p-1}(\phib+r_j)\right],\\ \mathbf{III} &= \re\int \phib_t \left[ {|\phib+r_j+z|}^{p-1}(\Bar\phib +\Bar{r_j}+\bar z) - {|\phib+r_j|}^{p-1}(\Bar\phib+\Bar{r_j}+\bar z) \right.\\ &\qquad\qquad \left. -(p-1){|\phib+r_j|}^{p-3}(\Bar\phib+\Bar{r_j}) \re[(\Bar\phib+\Bar{r_j})z] \right]. \end{aligned} \right. \] The purpose is now to make appear quadratic terms in $z$ in these expressions. For $\mathbf{II}$ and $\mathbf{III}$, we simply write \[ \mathbf{II} = -\re\int ih_1\bar z\left[ {|\phib+r_j|}^{p-1}z+(p-1)(\phib+r_j){|\phib+r_j|}^{p-3} \re[(\Bar\phib+\Bar{r_j})z] \right] +O(\nh{z}^3) \] and \begin{multline*} \mathbf{III} = \re\int \phib_t \left[ \left(\frac{p-1}{2}\right){|z|}^2 {|\phib+r_j|}^{p-3}(\Bar\phib+\Bar{r_j}) +(p-1)\bar z {|\phib+r_j|}^{p-3}\re[(\Bar\phib+\Bar{r_j})z] \right.\\ \left. +\frac{(p-1)(p-3)}{2}\Carre{\re[(\Bar\phib+\Bar{r_j})z]} {|\phib+r_j|}^{p-5}(\Bar\phib+\Bar{r_j}) \right] +O(\nh{z}^3). \end{multline*} For $\mathbf{I}$, we have to compute \begin{align*} \mathbf{I} &= \re\int \bar zh_2\left\{ (p-1){|\phib+r_j+z|}^{p-3} \re[{(\phib+r_j+z)}_x(\Bar\phib+\Bar{r_j}+\bar z)](\phib+r_j+z)\right. \\ &\qquad + {|\phib+r_j+z|}^{p-1}{(\phib+r_j+z)}_x -(p-1){|\phib+r_j|}^{p-3} \re[{(\phib+r_j)}_x(\Bar\phib+\Bar{r_j})](\phib+r_j)\\ &\qquad \left. -{|\phib+r_j|}^{p-1}{(\phib+r_j)}_x \right\}\\ &= \re\int \bar zh_2\left\{ (p-1)\re[{(\phib+r_j)}_x(\Bar\phib+\Bar{r_j})](\phib+r_j)\left[ {|\phib+r_j+z|}^{p-3} -{|\phib+r_j|}^{p-3}\right] \right.\\ &\qquad +(p-1)z{|\phib+r_j+z|}^{p-3}\re[{(\phib+r_j)}_x(\Bar\phib+\Bar{r_j}) +{(\phib+r_j)}_x\bar z+z_x(\Bar\phib+\Bar{r_j})+z_x\bar z]\\ &\qquad +(p-1)(\phib+r_j){|\phib+r_j+z|}^{p-3} \re[{(\phib+r_j)}_x\bar z+z_x(\Bar\phib+\Bar{r_j})+z_x\bar z]\\ &\qquad \left. +{(\phib+r_j)}_x\left[ {|\phib+r_j+z|}^{p-1} -{|\phib+r_j|}^{p-1}\right] + {|\phib+r_j+z|}^{p-1}z_x \right\}\\ &= \re\int \bar zh_2\left\{ (p-1)(p-3){|\phib+r_j|}^{p-5}(\phib+r_j)\re[{(\phib+r_j)}_x(\Bar\phib +\Bar{r_j})]\re[(\Bar\phib+\Bar{r_j})z] \right. \\ &\qquad +(p-1)z{|\phib+r_j|}^{p-3}\re[{(\phib+r_j)}_x(\Bar\phib+\Bar{r_j})]\\ &\qquad +(p-1)(\phib+r_j){|\phib+r_j|}^{p-3} \re[{(\phib+r_j)}_x\bar z+z_x(\Bar\phib +\Bar{r_j})]\\ &\qquad \left. + (p-1){(\phib+r_j)}_x{|\phib+r_j|}^{p-3}\re[(\Bar\phib+\Bar{r_j})z] +{|\phib+r_j|}^{p-1}z_x \right\} +O(\nh{z}^3). \end{align*} In the last expression, we integrate by parts the following two terms. First, we have \begin{align*} \re\int &\bar zh_2\cdot(p-1)(\phib+r_j){|\phib+r_j|}^{p-3} \re[{(\phib+r_j)}_x\bar z +z_x(\Bar\phib+\Bar{r_j})]\\ &= (p-1)\int \re[z(\Bar\phib+\Bar{r_j})] {\re[z(\Bar\phib+\Bar{r_j})]}_x h_2 {|\phib+r_j|}^{p-3}\\ &= -\left(\frac{p-1}{2}\right) \int \Carre{\re[(\Bar\phib+\Bar{r_j})z]} h_{2x} {|\phib+r_j|}^{p-3}\\ &\quad -\frac{(p-1)(p-3)}{2}\int \Carre{\re[(\Bar\phib+\Bar{r_j})z]} {|\phib+r_j|}^{p-5} h_2\re[{(\phib+r_j)}_x(\Bar\phib+\Bar{r_j})]. \end{align*} Second, we have similarly \begin{align*} \re\int \bar zh_2z_x& {|\phib+r_j|}^{p-1} = -\frac{1}{2}\int {|z|}^2\left[ h_{2x}{|\phib+r_j|}^{p-1} +h_2(p-1){|\phib+r_j|}^{p-3} \re[{(\phib+r_j)}_x(\Bar\phib+\Bar{r_j})] \right]\\ &= -\frac{1}{2}\int {|z|}^2h_{2x} {|\phib+r_j|}^{p-1} -\left(\frac{p-1}{2}\right) \re\int h_2{(\phib+r_j)}_x(\Bar\phib+\Bar{r_j}) {|\phib+r_j|}^{p-3} {|z|}^2. \end{align*}

Therefore, as $\nli{h_{2x}}\leq \frac{C}{\sqrt t}$, we have obtained \begin{align*} -\frac{1}{2}H'(t) &= O(e^{-(e_j+4\g)t}\nh{z}) +O(t^{-1/2}\nh{z}^2) +O(\nh{z}^3)\\ &\quad + \frac{(p-1)(p-3)}{2}\int \Carre{\re[(\Bar\phib+\Bar{r_j})z]} {|\phib+r_j|}^{p-5} h_2\re[{(\phib+r_j)}_x(\Bar\phib+\Bar{r_j})]\\ &\quad +\left(\frac{p-1}{2}\right) \int h_2{|z|}^2 {|\phib+r_j|}^{p-3} \re[{(\phib+r_j)}_x(\Bar\phib+\Bar{r_j})]\\ &\quad +(p-1)\re\int \bar zh_2 {(\phib+r_j)}_x {|\phib+r_j|}^{p-3} \re[(\Bar\phib+\Bar{r_j})z]\\ &\quad +\frac{(p-1)(p-3)}{2} \re\int \phib_t \Carre{\re[(\Bar\phib+\Bar{r_j})z]} {|\phib+r_j|}^{p-5}(\Bar\phib+\Bar{r_j})\\ &\quad + \left(\frac{p-1}{2}\right) \re\int \phib_t {|z|}^2 {|\phib+r_j|}^{p-3}(\Bar\phib+\Bar{r_j})\\ &\quad + (p-1)\re\int \phib_t \bar z {|\phib+r_j|}^{p-3}\re[(\Bar\phib+\Bar{r_j})z]\\ &\quad -(p-1)\re\int ih_1 \bar z (\phib+r_j){|\phib+r_j|}^{p-3} \re[(\Bar\phib+\Bar{r_j})z]. \end{align*} Finally, collecting similar terms in a single integral, we get, as $\nh{r_j}\leq Ce^{-e_jt}$, \begin{align*} -\frac{1}{2}H'(t) &= O(e^{-(e_j+4\g)t}\nh{z}) +O(t^{-1/2}\nh{z}^2) +O(\nh{z}^3)\\ &\quad + \frac{(p-1)(p-3)}{2}\re\int \Bar\phib {|\phib+r_j|}^{p-5} \Carre{\re[(\Bar\phib+\Bar{r_j})z]} \Big[ \phib_t +h_2\phib_x -ih_1\phib\Big]\\ &\quad +\left(\frac{p-1}{2}\right) \re\int {|z|}^2\Bar\phib {|\phib+r_j|}^{p-3} \Big[ \phib_t +h_2\phib_x -ih_1\phib\Big]\\ &\quad + (p-1)\re\int \bar z {|\phib+r_j|}^{p-3} \re[(\Bar\phib+\Bar{r_j})z] \Big[ \phib_t +h_2\phib_x -ih_1\phib\Big]\\ &= O(e^{-(e_j+4\g)t}\nh{z}) +O(t^{-1/2}\nh{z}^2) +O(\nh{z}^3), \end{align*} since ${\|\phib_t +h_2\phib_x -ih_1\phib\|}_{H^{-1}}\leq Ce^{-4\g t}$ by Lemma \ref{th:transport} and the three terms in front of $\phib_t +h_2\phib_x -ih_1\phib$ are bounded in $H^1$ by $\nh{z}^2$, which concludes the proof of Proposition~\ref{th:dHdtNLS}.
\end{proof}

\end{document}